\documentclass[journal]{IEEEtran}
\usepackage{amssymb}
\usepackage{amsthm}
\usepackage[cmex10]{amsmath}
\usepackage{graphicx}
\usepackage{colortbl}
\usepackage{cite}

\newtheorem{theorem}{Theorem}
\newtheorem{coro}{Corollary}
\newtheorem{remark}{Remark}
\newtheorem{lemma}{Lemma}
\newtheorem{assumption}{Assumption}
\newtheorem{definition}{Definition}
\newtheorem{problem}{Problem}

\ifCLASSINFOpdf
\else
\fi
\begin{document}
%
\title{Optimal Control and Stabilization for Networked Control Systems with Asymmetric Information}
%
%
%
\author{Xiao~Liang,~~Huanshui~Zhang~~and~~Juanjuan~Xu
\thanks{This work is supported by the Taishan Scholar Construction Engineering by Shandong Government, the National Natural Science Foundation of China (61120106011, 61633014, 61403235, 61573221).}
\thanks{J. Xu and H. Zhang are with School of Control Science and Engineering, Shandong University, Jinan, P.R.China. X. Liang is with College of Electrical Engineering and Automation, Shandong University of Science and Technology, Qingdao, Shandong, P.R.China 250061.}
}

\maketitle

\begin{abstract}
This paper considers the optimal control and stabilization problems for networked control systems (NCSs) with asymmetric information. In this NCSs model, the remote controller can receive packet-dropout states of the plant, and the available information for the embedded controller are observations of states and packet-dropout states sent from the remote controller. The two controllers operate the plant simultaneously to make the quadratic performance minimized and stabilize the linear plant. For the finite-horizon case, since states of the plant cannot be obtained perfectly, we develop the optimal estimators for the embedded and remote controllers based on asymmetric information respectively. Then we give the necessary and sufficient condition for the optimal control based on the solution to the forward-backward stochastic difference equations (FBSDEs). For the infinite-horizon case, on one hand, the necessary and sufficient condition is given for the stabilization in the mean-square sense of the system without the additive noise. On the other hand, it is shown that the system with the additive noise is bounded in the mean-square sense if and only if there exist the solutions to the two coupled algebraic Riccati equations. Numerical examples on the unmanned underwater vehicle are presented to show the effectiveness of the given algorithm.
\end{abstract}


%
\IEEEpeerreviewmaketitle

\section{Introduction}
Over the course of last few decades, advances in wireless communication have greatly boosted the development of networked control systems (NCSs). NCSs, containing the system, sensors, controllers and actuators where the operation is coordinated through a wireless communication, have attracted research interest due to its broad applications in electronic system, industrial manufacture and mobile communication \cite{R1,R2}. Comparing with the classical feedback control systems with wired point-to-point link, NCSs have been shown to be more cost-effective, provide higher flexibility and reduce the maintenance cost \cite{R3,R4}.

Recently, optimal control with asymmetric information (OCAI) has  received increasing attention due to the urgent demand in applications, such as deep-sea research, co-ordination of supply and demand, unmanned aerial vehicles and automated highway systems \cite{R5,R6,R7}. The so-called OCAI means that the system contains two or several controllers and the feedback information for different controllers are different. The fundamental difference between traditional optimal control (TOC) and OCAI is that for TOC the feedback information for different controllers are the same (same states or same observations) \cite{R8}. Thus, the method of TOC cannot be applied directly to deal with the problem of OCAI.

The research on TOC can be traced back to 50's in last century \cite{R9}. The stochastic optimal control (SOC) problem, pioneered by \cite{R10}, has gained continuous attention\cite{R11,R12,R13}. \cite{R11} considers the general case of SOC when the control weighing matrix and state weighing matrix of the performance are postive-definite and semi-positive definite respectively. \cite{R13} shows the solvability of the SOC problem by raising a generalized Riccati equation.

However, the above references consider the control problem with one or several control channels, and assume that the feedback information for different controllers are identical. This assumption hinders the development of the optimal control in applications. \cite{R14} gives the optimal strategy for the supplier and shows the impact of asymmetric disruption information on the performance of the supplier, the retailer and the supply chain. By solving forward-backward stochastic differential equations involved with two decoupled Riccati equations, \cite{R15} gives the respective optimal feedback strategies of both deterministic and random controllers for the linear stochastic system with asymmetric information. For the stochastic dynamic games with asymmetric information, \cite{R16} introduces the common information based perfect Bayesian equilibria and provides a sequential decomposition of the dynamic game.

Nevertheless, seldom work on NCS with asymmetric information has been investigated. Recently, \cite{R17} studies the NCSs with multiple local controllers and a remote controller where the information for local and remote controllers are different. The optimal control for the finite-horizon case has been solved in \cite{R17}. \cite{R18} considers both the optimal control for the finite-horizon case and the stabilization problem for the infinite-horizon case of the special model as in \cite{R17}. However, both \cite{R17} and \cite{R18} assume that the local controller can observe the state perfectly, which is not feasible in reality. Generally, the state is inevitably interfered with noises (multiplicative noise or additive noise) such that the controller cannot obtain the perfect state but the estimation of the state based on the received observations. Besides, the stabilization problem for NCSs with asymmetric information has not be solved completely so far. Since the controller cannot gain the perfect state, the control problem becomes more difficult and challenging.

In this paper, we consider the NCSs model containing a local device, a remote device and an unreliable communication channel as depicted in Fig. 1. The state transmits via two channels to the embedded controller and remote controller respectively. On one hand, the state $x_k$ is observed by the sensor and then sent to the estimator as the observation $y_k^P$. On the other hand, due to the unreliable communication channel, the state $x_k$ may be lost when transmits to the remote controller. The remote controller sends the received information $\{y_k^W,\ldots,,y_0^W,u_{k-1}^W,\ldots,u_0^W\}$ to the local device. The embedded controller makes its decision based on its own observations and observations of the remote controller, and the remote controller designs its action by using its own observation which results in the asymmetric information for the embedded and remote controllers. The embedded controller and the remote controller perform the plant simultaneously. The aim of this NCSs model is to minimize the quadratic performance and stabilize the plant. This NCSs model derives from increasing applications that request remote control of objects over wireless communication where the communication channels are tended to failure. Generally, the embedded controller may be an integrated chip on the local device with poor transmission capacity and the remote controller can be a mission-control center possessing powerful dispatching ability such that the link from the embedded controller to the remote controller is prone to failure and the negative link is perfect.
\begin{figure}[htbp]
  \begin{center}
  \includegraphics[width=0.48\textwidth]{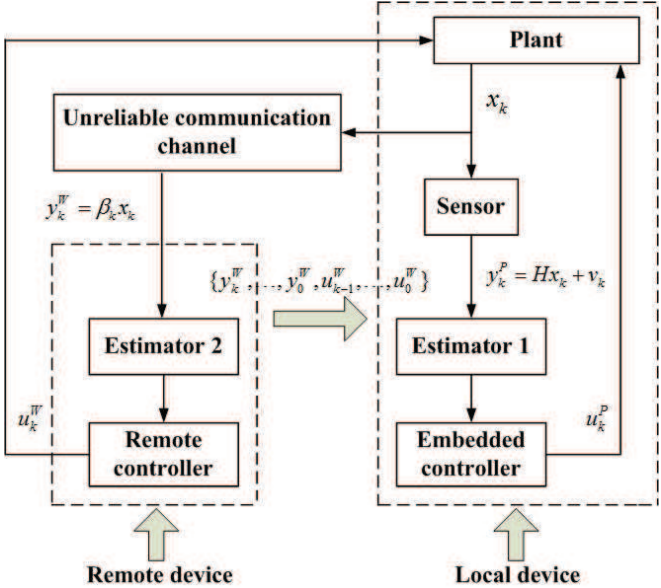}
  \caption{Overview of NCSs with asymmetric information.} \label{fig:digit}
  \end{center}
\end{figure}

In this paper, we shall focus on the optimal control and stabilization problems for NCSs with asymmetric information. Firstly, for the finite-horizon case, we show the optimal estimators for the two controllers respectively based on the asymmetric information. Then by applying the Pontryagin's maximum principle, a solution to the forward-backward stochastic difference equations (FBSDEs) is presented. Based on this solution, the optimal embedded and remote controllers are given. For the infinite-horizon case, by making use of the optimal performance of the finite-horizon case to define the Lyapunov function, we show the stabilization condition and the boundedness condition in the mean-square sense for the system with the additive noise and without the additive noise respectively in terms of two coupled algebraic Riccati equations. At last, we give numerical examples about the unmanned underwater vehicle to testify the effectiveness of the proposed algorithm.

The contributions of this paper are as follows:

(1) It is the first time to investigate and give the complete solution to the optimal control and stabilization problems for NCSs with asymmetric information where states of the plant cannot be obtained perfectly.

(2) For the finite-horizon case, we give the necessary and sufficient condition for the optimal control problem based on the solution to the FBSDEs.

(3) For the infinite-horizon case, the necessary and sufficient condition of the stabilization in the mean-square sense is presented for the system without the additive noise in terms of two coupled algebraic Riccati equations.

(4) We show the necessary and sufficient condition of the boundedness in the mean-square sense for the system with the additive noise. It should be emphasized that it is the first time to give the necessary and sufficient condition of the stabilization problem for linear quadratic gaussian (LQG) control when the system is involved with the additive noise.

The remainder of the paper is organized below. Section II presents the optimal estimators and optimal strategies for the embedded and remote controllers, respectively. The stabilization conditions for the system with additive noise and without the additive noise are given respectively in Section III. Section IV illustrates numerical examples on the unmanned underwater vehicle. The conclusion are given in Section V. The proofs of relevant results are in Appendices.

\emph{Notation:} Define $\mathbb{E}$ as the mathematical expectation operator. $\mathbb{R}^{n}$ presents the $n$-dimensional Euclidean space.  $tr(B)$ represents the trace of matrix $B$. Define $\{\mathcal{F}\{G_k\}\}$ as the natural filtration generated by the random variable $g_k$, i.e., $\mathcal{F}\{G_k\}=\sigma\{g_0,\ldots,g_k\}$. $B\geq0 (>0)$ denotes that $B$ is a positive semi-definite (positive definite) matrix. $I_{\{A\}}$ is an indicator function, i.e., $\varepsilon\in A$, $\mathcal{I}_{\{A\}}=1$, otherwise, $\mathcal{I}_{\{A\}}=0$. $|\lambda_{max}(A)|$ presents the eigenvalue of matrix $A$ with the largest absolute value.
\section{Optimal Control of NCSs}
\subsection{Problem Formulation}
The plant model of the system that is to be controlled takes the form of the discrete-time stochastic difference equation
\begin{align}
x_{k+1}&=Ax_k+B^Wu^W_k+B^Pu^P_k+\omega_k,\label{1}
\end{align}
where $x_k\in \mathbb{R}^n$ is the state, $u_k^P\in\mathbb{R}^P$ is the embedded controller and $u_k^W\in\mathbb{R}^W$ is the remote controller. $A, B^W, B^P$ are the constant matrices with appropriate dimensions. The initial value $x_0\in \mathbb{R}^n$ and $\omega_k\in\mathbb{R}^n$ are Gaussian and independent with  mean ($\mu,0$) and covariance ($\sigma,Q_\omega$).

The observed models for the two controllers are as follows:
\begin{align}
y^P_k&=Hx_k+v_k,\label{2}\\
y^W_k&=\beta_kx_k\label{3},
\end{align}
where $y^W_k\in \mathbb{R}^n$ is the observation for the remote controller and $y^P_k\in \mathbb{R}^m$ is the observation for the embedded controller. $v_k\in\mathbb{R}^m$ is the Gaussian white noise with zero mean and $Q_v$ covariance. $H$ is the constant matrix with appropriate dimension. $\beta_k$ is an independent identically distributed (i.i.d.) Bernoulli random variable presenting the signal transmission through the communication channel, i.e., $\beta_k=1$ signifies the successful transmission with probability $1-p$, and $\beta_k=0$ denotes the dropout of the packet with probability $p$.

The associate performance for the system (\ref{1}) is shown as
\begin{align}
\nonumber J_N&=\mathbb{E}\bigg\{\sum^N_{k=0}\bigg[x_k'Qx_k+u^{W'}_kR^Wu^W+u^{P'}_kR^Pu^P_k\bigg]\\
&\qquad\quad+x_{N+1}'P_{N+1}x_{N+1}\bigg\},\label{4}
\end{align}
where $R^W, R^P, Q$ and $P_{N+1}$ are positive semi-definite. $\mathbb{E}$ takes the mathematical expectation over the random processes $\{\beta_k\}$, $\{\omega_k\}$, $\{v_k\}$ and the random variable $x_0$.

As can be seen from Fig. 1, due to the limiting transmission capacity of the embedded device, the remote controller can merely obtain the observations $\{y^W_0,\ldots,y^W_k\}$ delivered from the local device. On the other hand, the embedded device has the information of itself observations $\{y^P_0,\ldots,y^P_k\}$ and the observations of the remote controller, i.e., $\{y^W_0,\ldots,y^W_k, u_0^W,\ldots,u_{k-1}^W\}$. For simplicity, we denote $\mathcal{F}\{Y^W_k\}$ as the $\sigma$-algebra generated by $\{y^W_0,\ldots,y^W_k\}$ and $\mathcal{F}\{Y^W_k,Y_k^P\}$ as the $\sigma$-algebra generated by $\{y^W_0,\ldots,y^W_k,y^P_0,\ldots,y^P_k\}$.

Then the problem to be solved in this section is formulated as follows:
\begin{problem}
Find the $\mathcal{F}\{Y^W_k\}$-measurable controller $u^W_k$ and the $\mathcal{F}\{Y^W_k,Y_k^P\}$-measurable controller $u^P_k$ such that the performance (\ref{4}) is minimized subject to the system (\ref{1}).
\end{problem}
\begin{remark}
Generally in practice, the state signal $x_k$ is always disturbed by the noise (multiplicative noise or additive noise) when obtained by the controller. In other word, the precise state $x_k$ cannot be acquired by the controller. Different from \cite{R17,R18} of receiving the precise state $x_k$ by the controller, this paper considers that the embedded controller $u^P_k$ cannot obtain the precise state $x_k$ but receive the observation $y^P_k$ which is more practical in application and becomes more difficult.
\end{remark}
\begin{remark}
Due to the existence of the asymmetric information for the two controllers, it is not available to augment the two controllers $u^W_k$ and $u^P_k$ as one controller $U_k$ and then use the traditional optimal control approach \cite{R8} to derive $U_k$.
\end{remark}
\begin{remark}
As can be seen from Fig. 1, the obtainable information for the embedded controller $u^P_k$ are $\{y_0^W,\ldots,y_k^W,y_0^P,\ldots,y_k^P,u_0^W,\ldots,u_{k-1}^W\}$ and for the remote controller $u^W_k$ are $\{y_0^W,\ldots,y_k^W\}$. Obviously, the embedded controller $u_k^P$ cannot use the present time decision of the remote controller $u_k^W$. In other word, the leader-follower approach \cite{R19} of computing $u^W_k$ firstly and then calculating $u_k^P$ based on the result of $u_k^W$, is not suitable. Similarly, the general optimal control strategies for two decision-makers, such as Nash equilibrium \cite{R20} and Stackelberg strategy \cite{R21}, are not appropriate in this paper.
\end{remark}
\subsection{Solution to Problem 1}
Before show the optimal strategies of this section, we shall provide the optimal estimators for the two controllers respectively.
\begin{lemma}
With observations $\{y^W_0,\ldots,y^W_k\}$ for the system (\ref{1}), the optimal estimator for the remote controller $u^W_k$ is presented as
\begin{align}
\hat{x}^W_{k|k}=\mathbb{E}[x_k|\mathcal{F}\{Y^W_k\}]=\gamma_kx_k+(1-\gamma_k)\hat{x}^W_{k|k-1},\label{9}
\end{align}
where $\gamma_k=\mathcal{I}_{\{y^W_k\neq0\}}$ with $P(\gamma_k=1)=1-p$, $\mathcal{I}$ denotes the indicator function and the initial value $\hat{x}^W_{0|-1}=\mu$.

Given observations $\{y^W_0,\ldots,y^W_k,y^P_0,\ldots,y^P_k\}$ for the system (\ref{1}), the optimal estimator for the embedded controller $u^P_k$ is given by
\begin{align}
\hat{x}^P_{k|k}=\mathbb{E}[x_k|\mathcal{F}\{Y^W_k,Y^P_k\}]=\gamma_kx_k+(1-\gamma_k)\hat{x}^{PP}_{k|k},\label{10}
\end{align}
where
\begin{align}
\nonumber\hat{x}^{PP}_{k|k}&=\mathbb{E}[x_k|\mathcal{F}\{Y^W_{k-1},Y^P_k\}]\\
                           &=\hat{x}^P_{k|k-1}+G^P_{k|k-1}(y^P_k-H\hat{x}^P_{k|k-1})\nonumber,
\end{align}
with
\begin{align}
G^P_{k|k-1}=\Sigma^P_{k|k-1}H'(H\Sigma^P_{k|k-1}H'+Q_v)^{-1},\nonumber
\end{align}
and the estimation error covariances
\begin{align}
\nonumber\Sigma^P_{k|k-1}&=\mathbb{E}[(x_k-\hat{x}^P_{k|k-1})(x_k-\hat{x}^P_{k|k-1})']\\
                &=A\Sigma^P_{k-1|k-1}A'+Q_\omega,\label{11}\\
\Sigma^P_{k|k}&=\mathbb{E}[(x_k-\hat{x}^P_{k|k})(x_k-\hat{x}^P_{k|k})']=p\Sigma^{PP}_{k|k},\label{12}\\
\nonumber\Sigma^{PP}_{k|k}&=\mathbb{E}[(x_k-\hat{x}^{PP}_{k|k})(x_k-\hat{x}^{PP}_{k|k})']\\
\nonumber                 &=(I-G^P_{k|k-1}H)\Sigma^P_{k|k-1}(I-G^P_{k|k-1}H)'\\
                          &\quad+G^P_{k|k-1}Q_vG^{P'}_{k|k-1},\nonumber
\end{align}
with the initial value $\hat{x}^P_{0|-1}=\mu$ and $\Sigma^P_{0|-1}=\sigma$.
\end{lemma}
\begin{proof}
The optimal estimator $\hat{x}^W_{k|k}$ can be obtained by similar procedures as in \cite{R22}. Now we shall show how to calculate the optimal estimator $\hat{x}^P_{k|k}$.

When $k=0$, the embedded controller can receive the observations $y^W_0$ and $y^P_0$. If $\beta_0=0$, the optimal strategy for the local device to estimate the state $x_0$ is to make use of the observation $y^P_0$. Thus, following the standard Kalman filtering, the optimal estimator is given by
\begin{align}
\nonumber\hat{x}^P_{0|0}&=\hat{x}^{PP}_{0|0}=\hat{x}^P_{0|-1}+G^P_{0|-1}(y^P_0-H\hat{x}^P_{0|-1}),\\
\nonumber \Sigma^{PP}_{0|0}&=(I\hspace{-0.8mm}-\hspace{-0.8mm}G^P_{0|-1}H)\Sigma^P_{0|-1}(I\hspace{-0.8mm}-\hspace{-0.8mm}G^P_{0|-1}H)'\hspace{-0.8mm}+\hspace{-0.8mm}G^P_{0|-1}Q_vG^{P'}_{0|-1},\nonumber
\end{align}
where
\begin{align}
\nonumber G^P_{0|-1}&=\Sigma^P_{0|-1}H'(H\Sigma^P_{0|-1}H'+Q_v)^{-1}.
\end{align}
If $\beta_0=1$, then the local device selects the observation $y^W_0$ to estimate the state $x_0$. Hence, the optimal estimator is as
$\hat{x}^P_{0|0}=x_0$. Thus, the optimal estimator (\ref{10}) holds for $k=0$.

When $k=1$, if $\beta_1=0$, then the embedded controller uses $\{y^P_1,y^P_0,y^W_0\}$ to estimate the state $x_1$. Then using the standard kalman filtering, the optimal estimator is given by
\begin{align}
\nonumber\hat{x}^P_{1|1}&=\hat{x}^{PP}_{1|1}=\hat{x}^P_{1|0}+G^P_{1|0}(y^P_1-H\hat{x}^P_{1|0}),\\
\nonumber \Sigma^{PP}_{1|1}&=(I\hspace{-0.8mm}-\hspace{-0.8mm}G^P_{1|0}H)\Sigma^P_{1|0}(I\hspace{-0.8mm}-\hspace{-0.8mm}G^P_{1|0}H)'\hspace{-0.8mm}+\hspace{-0.8mm}G^P_{1|0}Q_vG^{P'}_{1|0},\nonumber
\end{align}
where
\begin{align}
\nonumber G^P_{1|0}&=\Sigma^P_{1|0}H'(H\Sigma^P_{1|0}H'+Q_v)^{-1},\\
\nonumber \Sigma^P_{1|0}&=A\Sigma^P_{0|0}A'+Q_\omega.
\end{align}
If $\beta_1=1$, the local device applies $\{y^W_1,y^P_0,y^W_0\}$ to estimate the state $x_1$. Thus the optimal estimator is as
$\hat{x}^P_{1|1}=x_1.$
Hence, the estimator (\ref{10}) is valid for $k=1$.

Similarly, we can prove that the optimal estimator (\ref{10}) holds for $k=2,\ldots,N$. This ends the proof of Lemma 1.
\end{proof}
Following the similar discussion of \cite{R23}, we apply the Pontryagin's maximum principle to the system (\ref{1}) with the performance (\ref{4}) to yield the following costate equations:
\begin{align}
\lambda_{k-1}&=\mathbb{E}[A'\lambda_k+Qx_k|\mathcal{F}\{Y^W_k,Y^R_k\}],\label{new1}\\
0&=\mathbb{E}[B^{W'}\lambda_k|\mathcal{F}\{Y^W_k\}]+R^Wu^W_k,\label{new2}\\
0&=\mathbb{E}[B^{P'}\lambda_k|\mathcal{F}\{Y^W_k,Y^R_k\}]+R^Pu^P_k,\label{new3}\\
\lambda_{N}&=\mathbb{E}[P_{N+1}x_{N+1}|\mathcal{F}\{Y^W_k,Y^R_k\}],\label{new4}
\end{align}
where $\lambda_k$ is the costate, $k=0,\ldots,N$.
\begin{remark}
It can easily verify that by augmenting $u_k^W$ with $u_k^P$ as $U_k$ and making use of (\ref{new1})-(\ref{new4}), the remote controller $u_k^W$ can be readily obtained. Then substituting the result of $u_k^W$ into the system (\ref{1}) and using (\ref{new1}), (\ref{new3}) and (\ref{new4}), the embedded controller $u_k^P$ can be acquired. However, from Fig. 1 and Remark 3, the method of computing $u_k^W$ firstly and then calculating $u_k^P$ based on the results of $u_k^W$ is not valid in this paper. Thus, it is necessary to develop a novelty method of calculating the two controllers simultaneously.
\end{remark}
From Fig. 1, it can be observed that the embedded controller $u^P_k$ can receive observations $\{y^W_0,\ldots,y^W_k\}$. Now we make the following definition:
\begin{align}
u^P_k=\hat{u}^P_k+\tilde{u}^P_k,\label{5}
\end{align}
where $\hat{u}^P_k=\mathbb{E}[u^P_k|\mathcal{F}\{Y^W_k\}]$. Obviously, the following properties can be readily obtained:
\begin{align}
\nonumber&\mathbb{E}[\tilde{u}^P_k|\mathcal{F}\{Y^W_k,Y_k^P\}]=\tilde{u}^P_k,\mathbb{E}[\hat{u}^P_k|\mathcal{F}\{Y^W_k,Y_k^P\}]=\hat{u}^P_k,\\
&\mathbb{E}[\tilde{u}^P_k|\mathcal{F}\{Y^W_k\}]=0.\label{6}
\end{align}
By virtue of (\ref{5}), the system (\ref{1}) and the performance (\ref{4}) can be rewritten as
\begin{align}
x_{k+1}&=Ax_k+Bu_k+B^P\tilde{u}^P_k+\omega_k,\label{7}\\
\nonumber J_N&=\mathbb{E}\bigg\{\sum^N_{k=0}\bigg[x_k'Qx_k+u'_kRu_k+\tilde{u}^{P'}_kR^P\tilde{u}^P_k\bigg]\\
&\qquad\quad+x_{N+1}'P_{N+1}x_{N+1}\bigg\},\label{8}
\end{align}
where $u_k=\begin{bmatrix}u^W_k\\\hat{u}^P_k\end{bmatrix}$, $B=\begin{bmatrix}B^W&B^P\end{bmatrix}$ and $R=\begin{bmatrix}R^W&0\\0&R^P\end{bmatrix}$. Throughout this paper, we shall use the system (\ref{7}) and the performance (\ref{8}) instead of (\ref{1}) and (\ref{4}).

Based on the above transformation, we give the following lemma.
\begin{lemma}
Based on (\ref{5}) and (\ref{6}), we transform the costate equations (\ref{new1})-(\ref{new4}) into the following equations:
\begin{align}
\lambda_{k-1}&=\mathbb{E}[A'\lambda_k+Qx_k|\mathcal{F}\{Y^W_k,Y^R_k\}],\label{13}\\
0&=\mathbb{E}[B'\lambda_k|\mathcal{F}\{Y^W_k\}]+Ru_k,\label{14}\\
\nonumber0&=\mathbb{E}[B^{P'}\lambda_k|\mathcal{F}\{Y^W_k,Y^R_k\}]\\
          &\quad-\mathbb{E}[B^{P'}\lambda_k|\mathcal{F}\{Y^W_k\}]+R^P\tilde{u}^P_k,\label{15}\\
\lambda_{N}&=\mathbb{E}[P_{N+1}x_{N+1}|\mathcal{F}\{Y^W_k,Y^R_k\}],\label{16}
\end{align}
where $\lambda_k$ is the costate variable.
\end{lemma}
\begin{proof}
Taking mathematical expectation on both sides of (\ref{new3}) with $\mathcal{F}\{Y_k^W\}$ and using (\ref{13}), we get
\begin{align}
\nonumber0&=E\left[B^{P'}\hspace{-0.8mm}\lambda_k|\mathcal{F}\{Y_k^W\}\right]\hspace{-0.8mm}+\hspace{-0.8mm}E\left[R^Pu_k^P|\mathcal{F}\{Y_k^W\}\right]\\
          &=E\left[B^{P'}\hspace{-0.8mm}\lambda_k|\mathcal{F}\{Y_k^W\}\right]\hspace{-0.8mm}+R^P\hat{u}^P_k,\label{new5}
\end{align}
Combining (\ref{new2}) with (\ref{new5}), and noting (\ref{7}) and (\ref{8}), it yields
\begin{align}
\nonumber0=\hspace{-0.8mm}E\left[B'\lambda_k|\mathcal{F}\{Y_k^W\}\right]\hspace{-0.8mm}+Ru.
\end{align}
Subtracting (\ref{new5}) from (\ref{new3}), it yields that
\begin{align}
\nonumber           0&=E\left[B^{P'}\hspace{-0.8mm}\lambda_k|\mathcal{F}\{Y_k^W,Y_k^P\}\right]\hspace{-0.8mm}-\hspace{-0.8mm}E\left[B^{P'}\hspace{-0.8mm}\lambda_k|\mathcal{F}\{Y_k^W\}\right]\\
\nonumber&\quad+R^Pu^P-R^P\hat{u}_k^P\\
\nonumber&=E\left[B^{P'}\lambda_k|\mathcal{F}\{Y_k^W,Y_k^P\}\right]\\
\nonumber&\quad-E\left[B^{P'}\lambda_k|\mathcal{F}\{Y_k^W\}\right]+R^L\tilde{u}^P_k.
\end{align}
The proof has been completed.
\end{proof}
\begin{remark}
Through the transformation in Lemma 2, the two controllers $u_k$ and $\tilde{u}^P_k$ can be computed separately. In other word, we can calculate $\tilde{u}^P_k$ without using $u_k^W$. Please see the details in the following theorem.
\end{remark}
Now we are in the position to give the main results of this section.
\begin{theorem}
Problem 1 admits the unique solution if and only if $\Gamma_k$ and $\Omega_k$ are positive definite for $k=0,\ldots,N$.

In this case, the optimal controllers $u^W_k$ and $u^P_k$ are presented by
\begin{align}
u^W_k&=-\begin{bmatrix}I&0\end{bmatrix}\Gamma_k^{-1}M_k\hat{x}^W_{k|k},\label{17}\\
u^P_k&=-\begin{bmatrix}0&I\end{bmatrix}\Gamma_k^{-1}M_k\hat{x}^W_{k|k}-\Omega_k^{-1}L_k(\hat{x}^P_{k|k}-\hat{x}^W_{k|k}),\label{18}
\end{align}
where $\Gamma_k$, $M_k$, $\Omega_k$, $L_k$ and $\Delta_k$ obey
\begin{align}
\Gamma_k&=B'P_{k+1}^WB+R,\label{19}\\
M_k&=B'P_{k+1}^WA,\label{20}\\
\Omega_k&=B^{P'}\Delta_{k+1}B^P+R^P,\label{21}\\
L_k&=B^{P'}\Delta_{k+1}A,\label{22}\\
\Delta_k&=(1-p)P_k^W+pP_k^P,\label{23}
\end{align}
and $P_k^W$, $P_k^P$ satisfy the following coupled Riccati equations:
\begin{align}
P_k^W&=A'P_{k+1}^WA-M_k'\Gamma_k^{-1}M_k+Q,\label{24}\\
P_k^P&=A'\Delta_{k+1}A-L_k'\Omega_k^{-1}L_k+Q,\label{25}
\end{align}
with the terminal values $P_{N+1}^W=P_{N+1}^P=P_{N+1}$.

The optimal performance is given by
\begin{align}
\nonumber J_N^*&=\mathbb{E}\big\{x_0'\big[P_0^W\hat{x}_{0|0}^W+\hspace{-0.8mm}P_0^P(\hat{x}_{0|0}^P-\hspace{-0.8mm}\hat{x}_{0|0}^W)\big]\big\}\hspace{-0.8mm}+\hspace{-0.8mm}\sum_{k=0}^Ntr\bigg\{\Sigma_{k|k}^P\\
\nonumber&\quad\times[A'\Delta_{k+1}A\hspace{-0.8mm}+Q-p(A-G_{k+1|k}^PHA)'P_{k+1}^P(A\\
\nonumber&\quad-G_{k+1|k}^PHA)]\hspace{-0.8mm}+\hspace{-0.8mm}Q_\omega[(\Delta_{k+1}-p(I\hspace{-0.8mm}-G_{k+1|k}^PH)'P_{k+1}^P\\
\nonumber&\quad\times(I-G_{k+1|k}^PH)]-pQ_vG_{k+1|k}^{P'}P_{k+1}^PG_{k+1|k}^P\\
&\quad+\Sigma_{N+1|N+1}^PP_{N+1}\bigg\}.\label{26}
\end{align}
Moreover, the optimal costate $\lambda_{k-1}$ and estimators $\hat{x}_{k|k}^W$, $\hat{x}_{k|k}^P$ satisfy the following non-homogeneous relationship:
\begin{align}
\lambda_{k-1}=P_k^W\hat{x}_{k|k}^W+P_k^P(\hat{x}_{k|k}^P-\hat{x}_{k|k}^W).\label{27}
\end{align}
\begin{proof}
See Appendix A.
\end{proof}
\end{theorem}
\begin{remark}
It is noted that the non-homogeneous relationship (\ref{27}) is the solution to the FBSDEs (\ref{7}) and (\ref{13}). The key of obtaining the optimal strategy is to derive the non-homogeneous relationship (\ref{27}) and the maximum principle (\ref{13})-(\ref{16}), which are quite different from those of \cite{R18}.
\end{remark}
\section{Stabilization of NCSs}
In this section, the infinite horizon optimal control and stabilization problems will be solved. To make thoroughly study on the problems of the infinite horizon case, we shall proceed the research from two aspects, i.e., the system (\ref{7}) without the additive noise $\omega_k$ and with the additive noise $\omega_k$ respectively.
\begin{remark}
In fact, many references have investigated the stabilization problem for the system without the additive noise from several areas such as the minimum data rate \cite{R24} and the mean-square small gain \cite{R25}. It is noted that due to the existence of the additive noise, for the stabilization problem of the system with additive noise, only the boundedness in the mean square sense can be obtained \cite{R26}. In other words, the system cannot be stabilizable in the mean square sense in the presence of the additive noise. To derive a necessary and sufficient condition for the stabilization in the mean square sense, it is essential to study the system (\ref{7}) without the additive noise.
\end{remark}
\subsection{Stabilization in the Mean-Square Sense}
In this subsection, the system (\ref{7}) shall be written as the following equation:
\begin{align}
x_{k+1}&=Ax_k+Bu_k+B^P\tilde{u}^P_k.\label{28}
\end{align}
The associate infinite-horizon performance is given by
\begin{align}
J=&\mathbb{E}\sum_{k=0}^\infty[{x_k}'Qx_k+u_{k}'Ru_k+\tilde{u}_{k}^{P'}R^P\tilde{u}_{k}^P].\label{29}
\end{align}
We make some standard assumptions:
\begin{assumption}
$R^P>0$, $R^W>0$ and $Q=D'D\geq0$ for some matrices $D$.
\end{assumption}
\begin{assumption}
($A,Q^{\frac{1}{2}}$) is observable and ($A,H$) is detectable.
\end{assumption}
Before give the main results of this subsection, we present the following definitions:
\begin{definition}
The system (\ref{28}) with $u_k=0$ and $\tilde{u}_k^P=0$ is called asymptotically mean-square sense stable if the following equality
\begin{align}
\lim_{k\to\infty}\mathbb{E}(x_k'x_k)=0\nonumber
\end{align}
holds for any initial values $x_0$.
\end{definition}
\begin{definition}
The system (\ref{28}) is said to be stabilizable in the mean-square sense if there exist the $\mathcal{F}\{Y^W_k\}$-measurable $u_k=L^W\hat{x}_{k|k}^W$ and $\mathcal{F}\{Y_k^W,Y_k^P\}$-measurable $\tilde{u}_k^P=L^P(\hat{x}_{k|k}^P-\hat{x}_{k|k}^W)$ with constant matrices $L^W$ and $L^P$ such that for any $x_0$, the closed-loop system of (\ref{28}) is asymptotically mean-square stable.
\end{definition}
The problem to be dealt with in this subsection is presented below.
\begin{problem}
Find the $\mathcal{F}\{Y_k^W\}$-measurable $u_k$ and $\mathcal{F}\{Y^W_k,Y_k^P\}$-measurable $\tilde{u}_k^P$ such that the closed-loop system of (\ref{28}) is stabilizable in the mean-square sense and the infinite-horizon performance (\ref{29}) is minimized.
\end{problem}
Firstly, we show the convergence of the optimal estimators for the embedded controller and remote controller in the following lemma.
\begin{lemma}
Under Assumption 2, the estimation error covariance $\Sigma_{k|k}^P$ is convergent, i.e., $\lim_{k\to\infty}\Sigma_{k|k}^P=\Sigma^P$. Under Assumption 2, if $\sqrt{p}|\lambda_{max}(A\hspace{-0.8mm}-\hspace{-0.8mm}B^P\Omega^{-1}L)|<1$, then $\Sigma_{k|k}^W$ is convergent, i.e., $\lim_{k\to\infty}\Sigma_{k|k}^W=\Sigma^W$.
\end{lemma}
\begin{proof}
With (\ref{12}), we have
\begin{align}
\Sigma_{k|k}^P&=\sqrt{p}(I-G^P_{k|k-1}H)\Sigma^P_{k|k-1}(I-G^P_{k|k-1}H)'\sqrt{p}\nonumber\\
                          &\quad+\sqrt{p}G^P_{k|k-1}Q_vG^{P'}_{k|k-1}\sqrt{p}.\nonumber
\end{align}
Combining the above equation with (\ref{11}), it yields
\begin{align}
\nonumber\Sigma_{k+1|k}^P&=A\Sigma_{k|k}^PA'\\
\nonumber                 &=\sqrt{p}A(I-G^P_{k|k-1}H)\Sigma^P_{k|k-1}(I-G^P_{k|k-1}H)'A'\sqrt{p}\nonumber\\
                          &\quad+\sqrt{p}AG^P_{k|k-1}Q_vG^{P'}_{k|k-1}A'\sqrt{p}.\nonumber
\end{align}
Under Assumption 2, following the results of \cite{R27}, it can be obtained that $\lim_{k\to\infty}\Sigma_{k+1|k}^P=\tilde{\Sigma}^P$. Accordingly, we have that $\lim_{k\to\infty}\Sigma_{k|k}^P$ is convergent, i.e., $\lim_{k\to\infty}\Sigma_{k|k}^P=\Sigma^P$.

Using (\ref{9}), (\ref{28}) and (\ref{38}), we get
\begin{align}
\nonumber &x_k-\hat{x}_{k|k}^W\\
\nonumber&=(1-\gamma_k)(x_k-A\hat{x}_{k-1|k-1}^W-Bu_{k-1})\\
\nonumber&=(1-\gamma_k)[A(x_{k-1}-\hat{x}_{k-1|k-1}^W)+B^P\tilde{u}_{k-1}^P]\\
\nonumber&=(1-\gamma_k)[A(x_{k-1}-\hat{x}_{k-1|k-1}^W)\\
\nonumber&\qquad\qquad\quad-B^P\Omega^{-1}L(\hat{x}_{k-1|k-1}^P-\hat{x}_{k-1|k-1}^W)]\\
\nonumber&=(1\hspace{-0.8mm}-\hspace{-0.8mm}\gamma_k)[A(x_{k-1}\hspace{-0.8mm}-\hspace{-0.8mm}\hat{x}_{k-1|k-1}^W)\hspace{-0.8mm}+\hspace{-0.8mm}B^P\Omega^{-1}L(x_{k-1}\hspace{-0.8mm}-\hspace{-0.8mm}\hat{x}_{k-1|k-1}^P)\\
\nonumber&\qquad\qquad\quad-B^P\Omega^{-1}L(x_{k-1}-\hat{x}_{k-1|k-1}^W)].
\end{align}
Then the estimation error covariance $\Sigma_{k|k}^W$ can be calculated as
\begin{align}
\nonumber\Sigma_{k|k}^W&=p(A\Sigma_{k-1|k-1}^WA'\hspace{-0.8mm}+\hspace{-0.8mm}A\Sigma_{k-1|k-1}^PL'\Omega^{-1}B^{P'}\hspace{-0.8mm}-\hspace{-0.8mm}A\Sigma_{k-1|k-1}^W\\
\nonumber              &\qquad\times L'\Omega^{-1}B^{P'}\hspace{-0.8mm}+\hspace{-0.8mm}B^P\Omega^{-1}L\Sigma_{k-1|k-1}^PA'\hspace{-0.8mm}-\hspace{-0.8mm}B^P\Omega^{-1}L\\
\nonumber&\qquad\times\Sigma_{k-1|k-1}^WA'-B^P\Omega^{-1}L\Sigma_{k-1|k-1}^PL'\Omega^{-1}B^{P'}\\
\nonumber&\qquad+B^P\Omega^{-1}L\Sigma_{k-1|k-1}^WL'\Omega^{-1}B^{P'})\\
\nonumber&=\sqrt{p}(A-B^P\Omega^{-1}L)\Sigma_{k-1|k-1}^W(A-B^P\Omega^{-1}L)'\sqrt{p}\\
\nonumber&\qquad+p[B^P\Omega^{-1}L\Sigma_{k-1|k-1}^P(A-B^P\Omega^{-1}L)'\\
\nonumber&\qquad\qquad+A\Sigma_{k-1|k-1}^PL'\Omega^{-1}B^{P'}].
\end{align}
It is noted that $\lim_{k\to\infty}\Sigma_{k|k}^P=\Sigma^P$ under the Assumption 2. Thus it can be derived from the above equation that $\lim_{k\to\infty}\Sigma_{k|k}^W=\Sigma^W$ when $\sqrt{p}|\lambda_{max}(A-B^P\Omega^{-1}L)|<1$. This completes the proof of Lemma 2.
\end{proof}
\begin{theorem}
Under Assumptions 1 and 2, if the system (\ref{28}) is stabilizable in the mean-square sense, then the following algebraic Riccati equations (\ref{30}) and (\ref{31}) admit the solutions $P^W$ and $P^P$ satisfying $P^W>0$ and $\Delta>0$:
\begin{align}
P^W&=A'P^WA-M'\Gamma^{-1}M+Q,\label{30}\\
P^P&=A'\Delta A-L'\Omega^{-1}L+Q,\label{31}
\end{align}
where
\begin{align}
\Gamma&=B'P^WB+R,\label{32}\\
M&=B'P^WA,\label{33}\\
\Omega&=B^{P'}\Delta B^P+R^P,\label{34}\\
L&=B^{P'}\Delta A,\label{35}\\
\Delta&=(1-p)P^W+pP^P.\label{36}
\end{align}
\end{theorem}
\begin{proof}
See Appendix B.
\end{proof}
\begin{theorem}
Under Assumptions 1 and 2,  the system (\ref{28}) is stabilizable in the mean-square sense if and only if there exist solutions $P^W$ and $P^P$ to the algebraic Riccati equations (\ref{30}) and (\ref{31}) satisfying $P^W>0$ and $\Delta>0$.

In this case, the stabilizing controllers
\begin{align}
u_k&=-\Gamma^{-1}M\hat{x}_{k|k}^W,\label{37}\\
\tilde{u}_k^P&=-\Omega^{-1}L(\hat{x}_{k|k}^P-\hat{x}_{k|k}^W),\label{38}
\end{align}
also minimize the performance (\ref{29}). The optimal performance is given by
\begin{align}
\nonumber J^*&=\hspace{-0.8mm}\mathbb{E}[x_0'P^W\hat{x}_{0|0}^W\hspace{-0.8mm}+\hspace{-0.8mm}x_0'P^P(\hat{x}_{0|0}^P\hspace{-0.8mm}-\hspace{-0.8mm}\hat{x}_{0|0}^W)]\hspace{-0.8mm}+\hspace{-0.8mm}tr\hspace{-0.8mm}\sum_{i=0}^\infty\{\Sigma_{i|i}^P[(A'\Delta A\\
\nonumber&\quad\hspace{-0.8mm}+Q-(A-G_{i+1|i}^PHA)'P^P(A-G_{i+1|i}^PHA)]\\
&\quad-pQ_vG_{k+1|k}^{P'}P_{k+1}^PG_{k+1|k}^P\}\label{39}
\end{align}
\end{theorem}
\begin{proof}
See Appendix C.
\end{proof}
\begin{remark}
It is noted that the key of deriving the stabilization condition is to define the Lyapunov function (\ref{B.1}) which is more complicated than \cite{R18}.
\end{remark}
Now we shall show the other statement of the stabilization condition for the system (\ref{28}). Firstly, we give the following assumptions:
\begin{assumption}
$\left(A, \begin{bmatrix}B^W&B^P\end{bmatrix}\right)$ is stabilizable.
\end{assumption}
\begin{assumption}
($A, B^P$) is stabilizable and ($A,D$) is observable where $pQ+(1-p)P^W=DD'$.
\end{assumption}
\begin{lemma}
Under Assumptions 1-4, the coupled algebraic Riccati equations (\ref{30}) and (\ref{31})
admit the unique solutions $P^W$ and $P^P$ such that $P^W>0$ and $\Delta>0$.
\end{lemma}
\begin{proof}
Since the algebraic Riccati equation (\ref{30}) is the standard Riccati equation, under Assumptions 1 and 2, the proof of the uniqueness of $P^W$ can be found in \cite{R28}. Here we show the uniqueness of $P^P>0$ in (\ref{31}). Under Assumption 1, applying (\ref{36}), it yields that
\begin{align}
P^P=\frac{\Delta-(1-p)P^W}{p}.\label{40}
\end{align}
Using (\ref{31}), (\ref{35}) and (\ref{40}), it yields that
\begin{align}
\nonumber \frac{\Delta-(1-p)P^W}{p}&=A'\Delta A-A'\Delta B^P\Omega^{-1}B^{P'}\Delta A+Q.
\end{align}
Accordingly, we have
\begin{align}
\nonumber \Delta&=\sqrt{p}A'Q\sqrt{p}A-\sqrt{p}A'\Delta B^P\Omega^{-1}B^{P'}\Delta\sqrt{p}A+[pQ\\
\nonumber    &\quad+(1-p)P^W].
\end{align}
Noting \cite{R28}, if ($A, B^P$) is stabilizable and ($A,D$) is observable where $pQ+(1-p)P^W=DD'$, there exists the unique solution $\Delta>0$. Observing (\ref{36}), it is readily obtained that (\ref{31}) admits the unique solution $P^P$. This completes the proof.
\end{proof}
We now are ready to restate the stabilization condition of Theorem 3 as follows.
\begin{coro}
Under Assumptions 1-4, the system (\ref{28}) is stabilizable in the mean-square sense.
\end{coro}
\begin{proof}
Under Assumptions 3 and 4, noting Lemma 4, it can be known that the algebraic Riccati equations (\ref{30}) and (\ref{31}) admit the unique solutions $P^W$ and $P^P$ such that $P^W>0$ and $\Delta>0$. Thus, under Assumptions 1-4 and from Theorem 3, it is readily obtained that the system (\ref{28}) is stabilizable in the mean-square sense.
\end{proof}
\subsection{Boundedness in the Mean-Square Sense}
In this subsection, we shall show the stabilization condition for the system (\ref{7}).
\begin{remark}
It is noted that for the single-control system with the additive noise, merely the sufficient condition for the stabilization problem can be derived \cite{R29}. The necessary and sufficient stabilization condition is still unsolved. In the following subsection, we shall present the complete solution to the stabilization problem for the system with multiple controllers and additive noise.
\end{remark}
The associate infinite-horizon performance for the system (\ref{7}) is given by
\begin{align}
\tilde{J}\hspace{-0.8mm}&=\hspace{-0.8mm}\lim_{N\to\infty}\hspace{-0.8mm}\frac{1}{N}\bigg\{\mathbb{E}\sum_{k=0}^N\bigg[{x_k}'Qx_k\hspace{-0.8mm}+\hspace{-0.8mm}u_k'Ru_k\hspace{-0.8mm}+\hspace{-0.8mm}\tilde{u}_k^{P'}R^P\tilde{u}_k^P\bigg]\bigg\}.\label{41}
\end{align}
We give the problem to be solved in this subsection as follows:
\begin{problem}
Search the $\mathcal{F}\{Y_k^W\}$-measurable controller $u_k$ and the $\mathcal{F}\{Y_k^W,Y_k^P\}$-measurable controller $\tilde{u}_k^P$ such that the system (\ref{7}) is bounded in the mean-square sense and the infinite-horizon performance (\ref{41}) is minimized.
\end{problem}
Before give the main results of this subsection, we shall present the convergence of the estimators for the embedded controller and remote controller.
\begin{lemma}
Under Assumption 2, the estimation error covariances $\Sigma_{k|k}^P$ is asymptotic bounded, i.e., $\lim_{k\to\infty}\Sigma_{k|k}^P=\Sigma^P$. Under Assumption 2, if $\sqrt{p}|\lambda_{max}(A\hspace{-0.8mm}-\hspace{-0.8mm}B^P\Omega^{-1}L)|<1$, then $\Sigma_{k|k}^W$ is asymptotic bounded, i.e., $\lim_{k\to\infty}\Sigma_{k|k}^W=\Sigma^W$.
\end{lemma}
\begin{proof}
The proof is similar to that of Lemma 3. Thus is omit here.
\end{proof}
Now we shall show the main results of this subsection.
\begin{theorem}
Under assumption 1 and 2, if $\sqrt{p}|\lambda_{max}(A\hspace{-0.8mm}-\hspace{-0.8mm}B^P\Omega^{-1}L)|<1$, the system (\ref{7}) is bounded in the mean-square sense if and only if there exist solutions $P^W$ and $P^P$ to the algebraic Riccati equations (\ref{30}) and (\ref{31}) such that $P^W>0$ and $\Delta>0$.

Accordingly, the stabilizing controllers are as
\begin{align}
{u}_k&=-\Gamma^{-1}M\hat{x}_{k|k}^W,\label{42}\\
\tilde{u}_k^P&=-\Omega^{-1}L(\hat{x}_{k|k}^P-\hat{x}_{k|k}^W),\label{43}
\end{align}
and the optimal performance is minimized by the above controllers as
\begin{align}
\nonumber \tilde{J}^* &=tr\big\{\Sigma^P[(A'\Delta A\hspace{-0.8mm}+Q-(A-\hspace{-0.8mm}G^PHA)'P^P(A-\hspace{-0.8mm}G^PHA)]\\
\nonumber&\qquad+Q_\omega[(\Delta-p(I-G^PH)'P^P(I-G^PH)]\\
&\qquad-pQ_vG^{P'}P^PG^P\big\}.\label{44}
\end{align}
\end{theorem}
\begin{proof}
See Appendix D.
\end{proof}
\begin{remark}
It should be emphasized that it is the first time to show the strict proof for the necessary and sufficient stabilization condition of LQG control for the system involving with the additive noise.
\end{remark}
We now show the other claim of the stabilization condition for the system (\ref{7}).
\begin{coro}
Under Assumptions 1-4, if $\sqrt{p}|\lambda_{max}(A\hspace{-0.8mm}-\hspace{-0.8mm}B^P\Omega^{-1}L)|<1$, the system (\ref{28}) is bounded in the mean-square sense.
\end{coro}
\begin{proof}
The proof is similar to that of Corollary 1. Thus is omitted here.
\end{proof}
\section{Numerical Examples}
Recently, the control of autonomous unmanned underwater vehicle (AUUV) has gain increasing interests due to its extensive applications, such as deep-sea exploration, target tracking and precise striking [1], [2]. In this section, we shall investigate a simple AUUV system to illustrate the effectiveness of the proposed algorithm.

Consider a simple AUUV system including an unmanned underwater vehicle (UUV) and a mission-control center (MCC). Let $\xi_t$ and $\nu_t$ be the location and velocity of the UUV at time $t$ (it is assumed that the UUV sails in the straight line and the variables are one-dimensional for simplicity). Then, at time $t+1$, the location $\xi_{t+1}$ has the form as
\begin{align}
\xi_{t+1}=\xi_t+\nu_t+\theta_t,\label{50}
\end{align}
where $\theta_t$ stands for the disturbance during the navigation, e.g., undercurrent, and $\nu_t=\nu_t^P+\nu_t^W$ with $\nu_t^W$ being the imposed-velocity by the MCC and $\nu_t^P$ being the imposed-velocity of the UUV. The initial value $\xi_0$ and $\theta_t$ are Gaussian and independent, with mean ($\bar{\xi}_0,0$) and covariance ($\delta,Q_\theta$) respectively.

As can be seen in Fig. 2, the location $\xi_t$ delivered from the UUV to the MCC is prone to be lost with probability $p$ due to the limiting transmission capacity of the UUV. Then, the MCC sends the observed signals $f_t=\eta_t\xi_t$ ($\eta_t$ is the i.i.d. Bernoulli random variable, i.e., $\eta_t=1$ means the location transmits successfully, otherwise fails) to the UUV as well as the control mission. Since the MCC is generally full-equipped, the downlink from the MCC to the UUV is perfect. The UUV makes its own control action based on its own observations $\varphi_t$ ($\varphi_{t}=C\xi_t+\epsilon_t,$ where $\epsilon_t$ is the Gaussian white noise with zero mean and covariance $Q_\epsilon$, and $C$ is a constant) and the MCC's observations $f_t$. It is noted that the control action of the UUV and the control mission of the MCC perform on the UUV simultaneously.

The objective of the UUV system is to arrive at the destination (the location is $\tau$) and meanwhile the energy cost is minimized. To this end, we denote the above objective by the following performance
\begin{align}
 J_N=\sum_{t=0}^N\mathbb{E}[(\xi_t-\hspace{-0.8mm}\tau)'Q^c(\xi_t-\tau)+\hspace{-0.8mm}\nu_t^{P'}R^P\nu_t^P+\hspace{-0.8mm}\nu_t^{W'}R^W\nu_t^W],\label{51}
\end{align}
where the first term is the sum of quadratic distance between the real-time location and the destination, the second term is the sum of the quadratic real-time velocity, with $Q^c\geq0$, $R^P\geq0$ and $R^W\geq0$ being the weighting coefficients.
\begin{figure}[htbp]
  \begin{center}
  \includegraphics[width=0.48\textwidth]{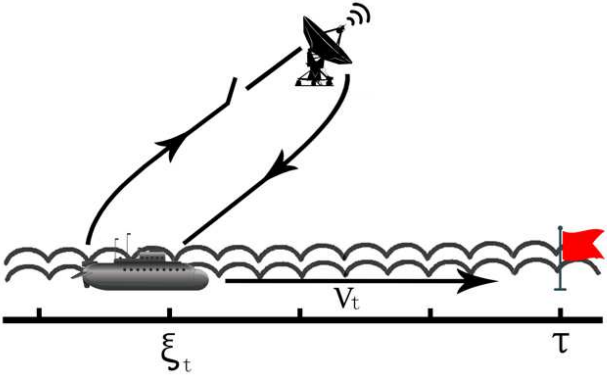}
  \caption{Over view of the AUUV system.} \label{fig:digit}
  \end{center}
\end{figure}

This AUUV system can be portrayed by the model of this paper. Denote $x_t=\xi_t-\tau$. Similar to (\ref{5}) and (\ref{6}), we define $\hat{\nu}_t^P=\mathbb{E}[\nu_t^P|\mathcal{F}\{f_t,\ldots,f_0\}]$, $\tilde{\nu}_t^P=\nu_t^P-\hat{\nu}_t^P$, $\hat{\nu}_t=\begin{bmatrix}\nu_t^W\\\hat{\nu}_t^P\end{bmatrix}$, $\hat{B}=\begin{bmatrix}1&1\end{bmatrix}$ and $R=\begin{bmatrix}R^W&0\\0&R^P\end{bmatrix}$. Then, the AUUV system (\ref{50}) can be rewritten as
\begin{align}
x_{t+1}=x_t+\hat{B}\hat{\nu}_t+\tilde{\nu}_t^P+\theta_t,\label{52}
\end{align}
The performance (\ref{51}) can be rewritten as
\begin{align}
J_N=\mathbb{E}\sum_{t=0}^N\{x_t'Q^cx_t+\hat{\nu}_t'R\hat{\nu}_t+\tilde{\nu}_t^{P'}R^P\tilde{\nu}_t^P\}.\label{53}
\end{align}
Comparing (\ref{7}), (\ref{8}) with (\ref{52}), (\ref{53}), the optimal strategies for the AUUV system can be obtained directly by applying Theorem 1 in Section II.

Set the system (\ref{52}) and the performance (\ref{53}) with $\bar{\xi}_0=0$, $\delta=Q_{\theta}=Q_\epsilon=1$, $\tau=30$, $C=1$, $Q^c=0.01$, $R^P=R^W=5$, $P_{N+1}=0$ and $N=100$.

To begin with, for the finite-horizon case, by applying Theorem 1, we draw Fig.3 and Fig. 4 as follows. Fig.3 shows the velocity of the UUV with $p=0$, $p=0.5$ and $p=1$ respectively. It is noted that there is little difference on the velocity of the UUV for different $p$. Fig. 4 presents the performance of the AUUV system for different $p$. It can be seen that the performance of the AUUV system becomes worse with the increasing of $p$.

For the infinite-horizon case, setting $p=0.5$, we shall firstly draw the curse of $\mathbb{E}[x_k'x_k]$ in Fig. 5 for the system (\ref{52}) without the additive noise $\theta_t$. It can be known that the regulated state is stable in the mean-square sense. Letting $p=0.6$, we draw the dynamic behavior of $\mathbb{E}[x_k'x_k]$ in Fig. 5 for the system (\ref{52}). It can be seen that the regulated state is bounded in the mean-square sense.
\begin{figure}[htbp]
  \begin{center}
  \includegraphics[width=0.38\textwidth]{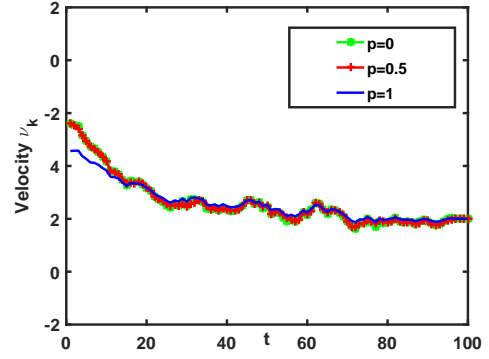}
  \caption{Velocity of the UUV for different $p$.} \label{fig:digit}
  \end{center}
\end{figure}
\begin{figure}[htbp]
  \begin{center}
  \includegraphics[width=0.38\textwidth]{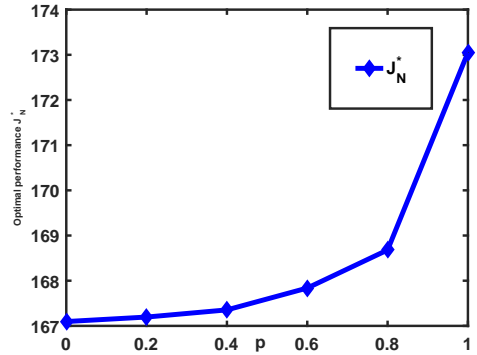}
  \caption{Optimal performance for different $p$.} \label{fig:digit}
  \end{center}
\end{figure}
\begin{figure}[htbp]
  \begin{center}
  \includegraphics[width=0.38\textwidth]{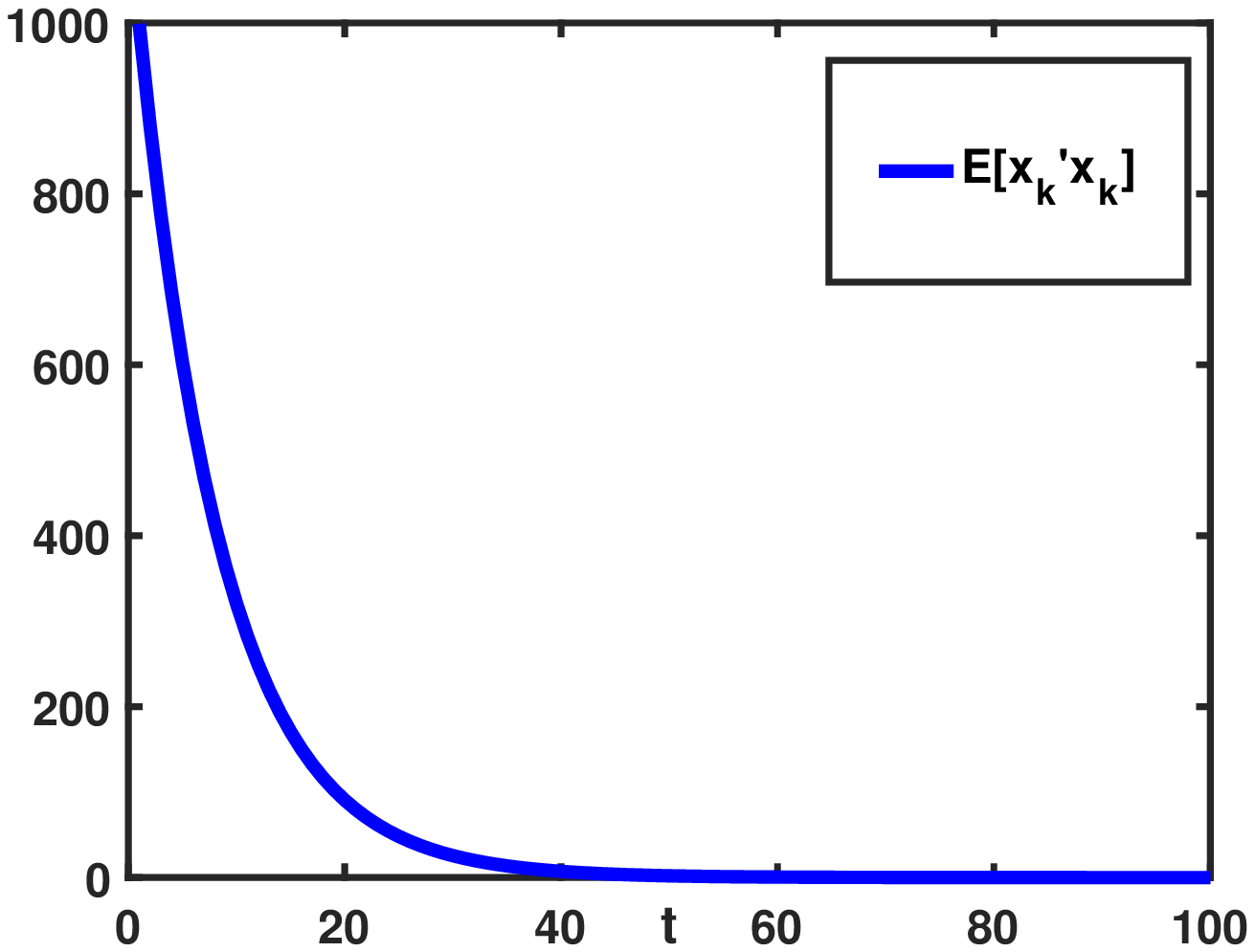}
  \caption{Dynamic Behavior of $E(x_k'x_k)$.} \label{fig:digit}
  \end{center}
\end{figure}
\begin{figure}[htbp]
  \begin{center}
  \includegraphics[width=0.38\textwidth]{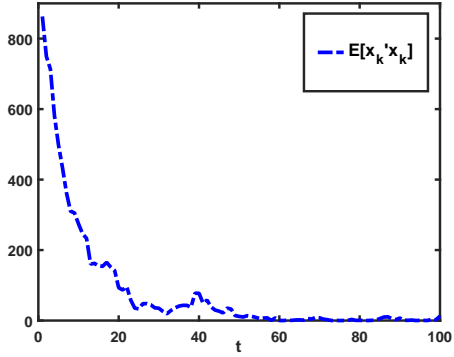}
  \caption{Dynamic Behavior of $E(x_k'x_k)$.} \label{fig:digit}
  \end{center}
\end{figure}
\section{Conclusion}
This paper is concerned about the optimal control and stabilization problem for NCSs with asymmetric information. We firstly present the optimal estimators by using the kalman filtering based on the asymmetric information. In virtue of the Pontryagin's maximum principle, we derive the solution to the FBSDEs. By applying this solution and two coupled Riccati equations, we give the optimal embedded and remote controllers. For the infinite-horizon case, based on the optimal performance, we define the Lyapunov function. In virtue of the Lyapunov function, the necessary and sufficient condition for the stabilization in the mean-square sense is given for the system without the additive noise. For the system with the additive noise, we show the necessary and sufficient condition for the boundedness in the mean-square sense in terms of two coupled algebraic Riccati equations. Finally, numerical examples about the unmanned underwater vehicle are shown.
\appendices
\section{Proof of Theorem 1}
\begin{proof}
``Sufficiency'': Suppose that $\Gamma_k$ and $\Omega_k$ are positive definite. We shall show that Problem 1 admits the unique solutions. By making use of (\ref{27}), denote the value function
\begin{align}
V_k=\mathbb{E}\big\{x_k'\big[P_k^W\hat{x}_{k|k}^W+P_k^P(\hat{x}_{k|k}^P-\hat{x}_{k|k}^W)\big]\big\}.\label{A.1}
\end{align}
Before proceeding the proof, we shall introduce the following preliminaries.

In virtue of (\ref{5}), (\ref{6}), (\ref{9}) and (\ref{10}), we get
\begin{align}
\nonumber&\hat{x}_{k+1|k+1}^W\\
\nonumber&=\gamma_{k+1}x_{k+1}+(1-\gamma_{k+1})\hat{x}_{k+1|k}^W\\
\nonumber                   &=\gamma_{k+1}(Ax_k\hspace{-0.8mm}+\hspace{-0.8mm}Bu_k\hspace{-0.8mm}+\hspace{-0.8mm}B^P\tilde{u}_k^P\hspace{-0.8mm}+\hspace{-0.8mm}\omega_k)\hspace{-0.8mm}+\hspace{-0.8mm}(1\hspace{-0.8mm}-\hspace{-0.8mm}\gamma_{k+1})(A\hat{x}_{k|k}^W\hspace{-0.8mm}+\hspace{-0.8mm}Bu_k)\\
                   &=\gamma_{k+1}A(x_k\hspace{-0.8mm}-\hspace{-0.8mm}\hat{x}_{k|k}^W)\hspace{-0.8mm}+\hspace{-0.8mm}\gamma_{k+1}B^P\tilde{u}_k^P+\hspace{-0.8mm}\gamma_{k+1}\omega_k+\hspace{-0.8mm}A\hat{x}_{k|k}^W+\hspace{-0.8mm}Bu_k,\label{A.2}
\end{align}
and
\begin{align}
\nonumber&\hat{x}_{k+1|k+1}^P\\
\nonumber&=\hspace{-0.8mm}\gamma_{k+1}x_{k+1}\hspace{-0.8mm}+\hspace{-0.8mm}(1\hspace{-0.8mm}-\hspace{-0.8mm}\gamma_{k+1})[\hat{x}^P_{{k+1}|k}\hspace{-0.8mm}+\hspace{-0.8mm}G^P_{{k+1}|k}(y^P_{k+1}\hspace{-0.8mm}-\hspace{-0.8mm}H\hat{x}^P_{{k+1}|k})]\\
\nonumber&=\gamma_{k+1}(Ax_k\hspace{-0.8mm}+\hspace{-0.8mm}Bu_k\hspace{-0.8mm}+\hspace{-0.8mm}B^P\tilde{u}_k^P\hspace{-0.8mm}+\hspace{-0.8mm}\omega_k)\hspace{-0.8mm}+\hspace{-0.8mm}(1\hspace{-0.8mm}-\hspace{-0.8mm}\gamma_{k+1})\{A\hat{x}_{k|k}^P\hspace{-0.8mm}+\hspace{-0.8mm}Bu_k\\
\nonumber&\quad+B^P\tilde{u}_k^P\hspace{-0.8mm}+\hspace{-0.8mm}G_{k+1|k}^P[H(Ax_k\hspace{-0.8mm}+\hspace{-0.8mm}Bu_k+\hspace{-0.8mm}B^P\tilde{u}_k^P+\omega_k)+\hspace{-0.8mm}v_{k+1}\\
\nonumber&\quad-H(A\hat{x}_{k|k}^P+Bu_k+B^P\tilde{u}_k^P)]\}\\
\nonumber&=\gamma_{k+1}A(x_k-\hat{x}_{k|k}^P)+A\hat{x}_{k|k}^P+Bu_k+B^P\tilde{u}_k^P+\gamma_{k+1}\omega_k\\
&\quad+(1-\hspace{-0.8mm}\gamma_{k+1})G_{k+1|k}^P[HA(x_k-\hspace{-0.8mm}\hat{x}_{k|k}^P)\hspace{-0.8mm}+H\omega_k\hspace{-0.8mm}+\hspace{-0.8mm}v_{k+1}].\label{A.3}
\end{align}
Combining (\ref{A.2}) with (\ref{A.3}), it yields that
\begin{align}
\nonumber&\hat{x}_{k+1|k+1}^P-\hat{x}_{k+1|k+1}^W\\
\nonumber&=(1-\gamma_{k+1})A(\hat{x}_{k|k}^P-\hat{x}_{k|k}^W)+(1-\gamma_{k+1})B^P\tilde{u}_k^P\\
&+(1-\gamma_{k+1})G_{k+1|k}^P[HA(x_k-\hat{x}_{k|k}^P)+H\omega_k+v_{k+1}].\label{A.4}
\end{align}
By applying (\ref{A.1}), (\ref{6}) and the orthogonality principle, it yields
\begin{align}
\nonumber &V_k\\
\nonumber&=\mathbb{E}[x_k'P_k^W\hat{x}_{k|k}^W+x_k'P_k^P(\hat{x}_{k|k}^P-\hat{x}_{k|k}^W)]\\
\nonumber&=\mathbb{E}[x_k'P_k^Wx_k\hspace{-0.8mm}-x_k'P_k^W(x_k-\hspace{-0.8mm}\hat{x}_{k|k}^W)]\hspace{-0.8mm}+\hspace{-0.8mm}\mathbb{E}\{[(x_k-\hspace{-0.8mm}\hat{x}_{k|k}^W)+\hspace{-0.8mm}\hat{x}_{k|k}^W]'\\
\nonumber&\quad\times P_k^P[(x_k-\hat{x}_{k|k}^W)-(x_k-\hat{x}_{k|k}^P)]\}\\
\nonumber&=\mathbb{E}[x_k'P_k^Wx_k-(x_k-\hspace{-0.8mm}\hat{x}_{k|k}^W)'P_k^W(x_k-\hspace{-0.8mm}\hat{x}_{k|k}^W)]\\
\nonumber&\quad+\hspace{-0.8mm}\mathbb{E}[(x_k\hspace{-0.8mm}-\hspace{-0.8mm}\hat{x}_{k|k}^W)'P_k^P(x_k-\hspace{-0.8mm}\hat{x}_{k|k}^W)\hspace{-0.8mm}-\hspace{-0.8mm}(x_k\hspace{-0.8mm}-\hspace{-0.8mm}\hat{x}_{k|k}^W)'P_k^P(x_k\hspace{-0.8mm}-\hspace{-0.8mm}\hat{x}_{k|k}^P)]\\
\nonumber&=\hspace{-0.8mm}\mathbb{E}[x_k'P_k^Wx_k\hspace{-0.8mm}-\hspace{-0.8mm}(\hat{x}_{k|k}^P\hspace{-0.8mm}-\hspace{-0.8mm}\hat{x}_{k|k}^W)'P_k^W(\hat{x}_{k|k}^P\hspace{-0.8mm}-\hspace{-0.8mm}\hat{x}_{k|k}^W)\hspace{-0.8mm}-\hspace{-0.8mm}tr(\Sigma_{k|k}^PP_k^W)]\\
\nonumber&\quad+\mathbb{E}[(\hat{x}_{k|k}^P\hspace{-0.8mm}-\hspace{-0.8mm}\hat{x}_{k|k}^W)'P_k^P(\hat{x}_{k|k}^P-\hspace{-0.8mm}\hat{x}_{k|k}^W)]\\
\nonumber&=\mathbb{E}[x_k'P_k^Wx_k\hspace{-0.8mm}+(\hat{x}_{k|k}^P\hspace{-0.8mm}-\hspace{-0.8mm}\hat{x}_{k|k}^W)'(P_k^P-P_k^W)(\hat{x}_{k|k}^P-\hspace{-0.8mm}\hat{x}_{k|k}^W)]\\
&\quad-tr(\Sigma_{k|k}^PP_k^W).\label{A.5}
\end{align}
In virtue of (\ref{A.1}), (\ref{6}), (\ref{7}), (\ref{A.2})-(\ref{A.4}), we have
\begin{align}
\nonumber&V_{k+1}\\
\nonumber&=\mathbb{E}\{x_{k+1}'[P_{k+1}^W\hat{x}_{k+1|k+1}^W+P_{k+1}^P(\hat{x}_{k+1|k+1}^P-\hat{x}_{k+1|k+1}^W)]\}\\
\nonumber&=\mathbb{E}\big\{(1\hspace{-0.8mm}-\hspace{-0.8mm}p)x_k'A'P_{k+1}^WA(x_k\hspace{-0.8mm}-\hspace{-0.8mm}\hat{x}_{k|k}^W)\hspace{-0.8mm}+\hspace{-0.8mm}(1\hspace{-0.8mm}-\hspace{-0.8mm}p)x_k'A'P_{k+1}^WB^P\tilde{u}_k^P\\
\nonumber&\quad+x_k'A'P_{k+1}^WA\hat{x}_{k|k}^W+x_k'A'P_{k+1}^WBu_k+u_k'B'P_{k+1}^WA\hat{x}_{k|k}^W\\
\nonumber&\quad+u_k'B'P_{k+1}^WBu_k+(1-p)\tilde{u}_k^{P'}B^{P'}P_{k+1}^WA(x_k-\hat{x}_{k|k}^W)\\
\nonumber&\quad+(1-p)\tilde{u}_k^{P'}B^{P'}P_{k+1}^WB^P\tilde{u}_k^P+(1-p)tr(Q_\omega P_{k+1}^W)\\
\nonumber&\quad-p[(x_k-\hspace{-0.8mm}\hat{x}_{k|k}^P)'(A-\hspace{-0.8mm}G_{k+1|k}^PHA)'P_{k+1}^P(A-\hspace{-0.8mm}G_{k+1|k}^PHA)\\
\nonumber&\quad\times(x_k-\hat{x}_{k|k}^P)]-ptr[Q_vG_{k+1|k}^{P'}P_{k+1}^PG_{k+1|k}^P+Q_\omega(I\\
\nonumber&\quad-G_{k+1|k}^PH)'P_{k+1}^P(I-\hspace{-0.8mm}G_{k+1|k}^PH)]+p[x_k'A'P_{k+1}^PA(x_k\\
\nonumber&\quad-\hat{x}_{k|k}^W)+x_k'A'P_{k+1}^PB^P\tilde{u}_k^P+\tilde{u}_k^{P'}B^{P'}P_{k+1}^PA(x_k-\hat{x}_{k|k}^W)\\
\nonumber&\quad+\tilde{u}_k^{P'}B^{P'}P_{k+1}^PB^P\tilde{u}_k^P+tr(Q_\omega P_{k+1}^P)]\big\}\\
\nonumber&=\mathbb{E}\big\{(1-p)\big[(\hat{x}_{k|k}^P-\hat{x}_{k|k}^W)'A'P_{k+1}^WA(\hat{x}_{k|k}^P-\hat{x}_{k|k}^W)\\
\nonumber&\quad+tr(\Sigma_{k|k}^PA'P_{k+1}^WA)+2(\hat{x}_{k|k}^P-\hat{x}_{k|k}^W)'A'P_{k+1}^WB^P\tilde{u}_k^P\\
\nonumber&\quad+\tilde{u}_k^{P'}B^{P'}P_{k+1}^WB^P\tilde{u}_k^P+\hspace{-0.8mm}tr(Q_\omega P_{k+1}^W)\big]\hspace{-0.8mm}+\hspace{-0.8mm}p\big[(\hat{x}_{k|k}^P-\hat{x}_{k|k}^W)'\\
\nonumber&\quad\times A'P_{k+1}^PA(\hat{x}_{k|k}^P\hspace{-0.8mm}-\hspace{-0.8mm}\hat{x}_{k|k}^W)\hspace{-0.8mm}+\hspace{-0.8mm}2(\hat{x}_{k|k}^P-\hat{x}_{k|k}^W)'A'P_{k+1}^PB^P\tilde{u}_k^P\\
\nonumber&\quad+\tilde{u}_k^{P'}B^{P'}P_{k+1}^PB^P\tilde{u}_k^P\big]+x_k'A'P_{k+1}^WAx_k-\hspace{-0.8mm}(\hat{x}_{k|k}^P\hspace{-0.8mm}-\hspace{-0.8mm}\hat{x}_{k|k}^W)'\\
\nonumber&\quad\times A'P_{k+1}^WA(\hat{x}_{k|k}^P\hspace{-0.8mm}-\hspace{-0.8mm}\hat{x}_{k|k}^W)+2u_k'B'P_{k+1}^WA\hat{x}_{k|k}^W\hspace{-0.8mm}+\hspace{-0.8mm}u_k'B'P_{k+1}^W\\
\nonumber&\quad\times Bu_k-ptr[\Sigma_{k|k}^P(A-\hspace{-0.8mm}G_{k+1|k}^PHA)'P_{k+1}^P(A-\hspace{-0.8mm}G_{k+1|k}^PH\\
\nonumber&\quad\times A)]-ptr[Q_vG_{k+1|k}^{P'}P_{k+1}^PG_{k+1|k}^P+Q_\omega(I-G_{k+1|k}^PH)'\\
&\quad\times P_{k+1}^P(I-\hspace{-0.8mm}G_{k+1|k}^PH)-Q_\omega P_{k+1}^P]\big\}.\label{A.6}
\end{align}
Combining (\ref{A.5}) with (\ref{A.6}) and using (\ref{19})-(\ref{25}), we get
\begin{align}
\nonumber&V_k-V_{k+1}\\
\nonumber&=\mathbb{E}\big\{\hspace{-0.8mm}x_k'(P_k^W\hspace{-0.8mm}-\hspace{-0.8mm}A'P_{k+1}^WA\hspace{-0.8mm}+\hspace{-0.8mm}M_k'\Gamma_k^{-1}M_k)x_k\hspace{-0.8mm}-\hspace{-0.8mm}x_k'M_k'\Gamma_k^{-1}M_kx_k\\
\nonumber&\quad+(\hat{x}_{k|k}^P\hspace{-0.8mm}-\hspace{-0.8mm}\hat{x}_{k|k}^W)'[P_k^P-P_k^W+pA'P_{k+1}^WA-pA'P_{k+1}^PA]\\
\nonumber&\quad\times(\hat{x}_{k|k}^P\hspace{-0.8mm}-\hspace{-0.8mm}\hat{x}_{k|k}^W)-2\tilde{u}_k^{P'}B^{P'}[pP_{k+1}^P+(1-p)P_{k+1}^W]A(\hat{x}_{k|k}^P\hspace{-0.8mm}\\
\nonumber&\quad-\hspace{-0.8mm}\hat{x}_{k|k}^W)-\tilde{u}_k^{P'}(\Omega_k-R^R)\tilde{u}_k^P-2u_k'B'P_{k+1}^WA\hat{x}_{k|k}^W\\
\nonumber&\quad-u_k'(\Gamma_k\hspace{-0.8mm}-\hspace{-0.8mm}R)u_k\big\}\hspace{-0.8mm}-\hspace{-0.8mm}tr\big\{Q_\omega[(1\hspace{-0.8mm}-\hspace{-0.8mm}p)P_{k+1}^W\hspace{-0.8mm}-\hspace{-0.8mm}p(I\hspace{-0.8mm}-\hspace{-0.8mm}G_{k+1|k}^PH)'\\
\nonumber&\quad\times P_{k+1}^P(I\hspace{-0.8mm}-\hspace{-0.8mm}G_{k+1|k}^PH)\hspace{-0.8mm}+\hspace{-0.8mm}pP_{k+1}^P]\hspace{-0.8mm}-\hspace{-0.8mm}pQ_vG_{k+1|k}^{P'}P_{k+1}^PG_{k+1|k}^P\\
\nonumber&\quad+\Sigma_{k|k}^P[-pA'P_{k+1}^WA+P_k^W+M_k'\Gamma_k^{-1}M_k\\
\nonumber&\quad+p(A-\hspace{-0.8mm}G_{k+1|k}^PHA)'P_{k+1}^P(A-\hspace{-0.8mm}G_{k+1|k}^PHA)]\big\}\\
\nonumber&=\mathbb{E}\big\{x_k'Qx_k+u_k'Ru_k+\tilde{u}_k^{P'}R^P\tilde{u}_k^P-(2u_k'B'P_{k+1}^WA\hat{x}_{k|k}^W\\
\nonumber&\quad+u_k'\Gamma_ku_k+\hat{x}_{k|k}^WM_k'\Gamma_k^{-1}M_k\hat{x}_{k|k}^W)-\big[2\tilde{u}_k^{P'}B^{P'}\Delta_{k+1}A\\
\nonumber&\quad\times(\hat{x}_{k|k}^P\hspace{-0.8mm}-\hspace{-0.8mm}\hat{x}_{k|k}^W)+\tilde{u}_k^{P'}\Omega_k\tilde{u}_k^P-(\hat{x}_{k|k}^P\hspace{-0.8mm}-\hspace{-0.8mm}\hat{x}_{k|k}^W)'(P_k^P-P_k^W\\
\nonumber&\quad+pA'P_{k+1}^WA-pA'P_{k+1}^PA-M_k'\Gamma_k^{-1}M_k)(\hat{x}_{k|k}^P\hspace{-0.8mm}-\hspace{-0.8mm}\hat{x}_{k|k}^W)\big]\big\}\\
\nonumber&\quad-tr\big\{\Sigma_{k|k}^P[A'\Delta_{k+1}A\hspace{-0.8mm}+Q-p(A-G_{k+1|k}^PHA)'P_{k+1}^P(A\\
\nonumber&\quad-G_{k+1|k}^PHA)]\hspace{-0.8mm}+\hspace{-0.8mm}Q_\omega[(\Delta_{k+1}-p(I\hspace{-0.8mm}-G_{k+1|k}^PH)'P_{k+1}^P\\
\nonumber&\quad\times(I-G_{k+1|k}^PH)]-pQ_vG_{k+1|k}^{P'}P_{k+1}^PG_{k+1|k}^P\big\}
\end{align}
\begin{align}
\nonumber&=\mathbb{E}\big\{x_k'Qx_k+u_k'Ru_k+\tilde{u}_k^{P'}R^P\tilde{u}_k^P-(u_k+\Gamma_k^{-1}M_k\hat{x}_{k|k}^W)'\\
\nonumber&\quad\times\Gamma_k(u_k+\Gamma_k^{-1}M_k\hat{x}_{k|k}^W)-[\tilde{u}_k^P+\Omega_k^{-1}L_k(\hat{x}_{k|k}^P\hspace{-0.8mm}-\hspace{-0.8mm}\hat{x}_{k|k}^W)]'\\
\nonumber&\quad\times\Omega_k[\tilde{u}_k^P+\Omega_k^{-1}L_k(\hat{x}_{k|k}^P\hspace{-0.8mm}-\hspace{-0.8mm}\hat{x}_{k|k}^W)]\big\}\\
\nonumber&\quad-tr\big\{\Sigma_{k|k}^P[A'\Delta_{k+1}A\hspace{-0.8mm}+Q-p(A-G_{k+1|k}^PHA)'P_{k+1}^P(A\\
\nonumber&\quad-G_{k+1|k}^PHA)]\hspace{-0.8mm}+\hspace{-0.8mm}Q_\omega[(\Delta_{k+1}-p(I\hspace{-0.8mm}-G_{k+1|k}^PH)'P_{k+1}^P\\
\nonumber&\quad\times(I-G_{k+1|k}^PH)]-pQ_vG_{k+1|k}^{P'}P_{k+1}^PG_{k+1|k}^P\big\}.
\end{align}
Adding from $k=0$ to $k=N$ on both sides of the above equation, the performance (\ref{8}) can be written as
\begin{align}
\nonumber&J_N\\
\nonumber&=\mathbb{E}\big\{x_0'\big[P_0^W\hat{x}_{0|0}^W+\hspace{-0.8mm}P_0^P(\hat{x}_{0|0}^P-\hspace{-0.8mm}\hat{x}_{0|0}^W)\big]\big\}\hspace{-0.8mm}+\hspace{-0.8mm}\sum_{k=0}^Ntr\bigg\{\Sigma_{k|k}^P\\
\nonumber&\quad\times[(A'\Delta_{k+1}A\hspace{-0.8mm}+Q-p(A-G_{k+1|k}^PHA)'P_{k+1}^P(A\\
\nonumber&\quad-G_{k+1|k}^PHA)]\hspace{-0.8mm}+\hspace{-0.8mm}Q_\omega[(\Delta_{k+1}-p(I\hspace{-0.8mm}-G_{k+1|k}^PH)'P_{k+1}^P\\
\nonumber&\quad\times(I-G_{k+1|k}^PH)]-pQ_vG_{k+1|k}^{P'}P_{k+1}^PG_{k+1|k}^P\\
\nonumber&\quad+\Sigma_{N+1|N+1}^PP_{N+1}\bigg\}+\sum_{k=0}^N\{(u_k+\Gamma_k^{-1}M_k\hat{x}_{k|k}^W)'\Gamma_k(u_k\\
\nonumber&\quad+\Gamma_k^{-1}M_k\hat{x}_{k|k}^W)+\hspace{-0.8mm}[\tilde{u}_k^P\hspace{-0.8mm}+\hspace{-0.8mm}\Omega_k^{-1}L_k(\hat{x}_{k|k}^P\hspace{-0.8mm}-\hspace{-0.8mm}\hat{x}_{k|k}^W)]'\Omega_k\\
\nonumber&\quad\times[\tilde{u}_k^P\hspace{-0.8mm}+\hspace{-0.8mm}\Omega_k^{-1}L_k(\hat{x}_{k|k}^P\hspace{-0.8mm}-\hspace{-0.8mm}\hat{x}_{k|k}^W)\}
\end{align}
Note that $\Gamma_k>0$ and $\Omega_k>0$ for $k=0,\ldots,N$. Thus, the optimal controllers are given by (\ref{17}) and (\ref{18}). Accordingly, the optimal performance is as (\ref{26}). This ends the proof of the sufficiency.

``Necessity'': The proof of the necessity is similar to that of \cite{R18}. Thus we omit here. We shall show that (\ref{27}) holds for $k=N+1,\ldots,0$ by mathematical induction.

Firstly, with (\ref{16}) and $P^W_{N+1}=P^P_{N+1}=P_{N+1}$, it is readily obtained that (\ref{27}) holds for $k=N+1$.

For $k=N$, using (\ref{7}), (\ref{6}) and (\ref{16}), we have (\ref{14}) as
\begin{align}
\nonumber0&=B'P_{N+1}(A\hat{x}_{N|N}^W+Bu_N)+Ru_N.
\end{align}
Thus, with (\ref{19}) and (\ref{20}), the optimal $u_N$ is presented as
\begin{align}
u_N=-\Gamma_N^{-1}M_N\hat{x}_{N|N}^W.\label{A.7}
\end{align}
In virtue of (\ref{7}), (\ref{6}) and (\ref{16}), (\ref{15}) can be calculated by
\begin{align}
\nonumber0&=B^{P'}P_{N+1}(A\hat{x}_{N|N}^P+Bu_N+B^{P'}\tilde{u}_N^{P'})\\
\nonumber &\quad-B^{P'}P_{N+1}(A\hat{x}_{N|N}^W+Bu_N)+R^P\tilde{u}_k^P.
\end{align}
Hence, using (\ref{21})-(\ref{23}), the optimal $\tilde{u}_N^P$ is given by
\begin{align}
\tilde{u}_N^P=-\Omega_N^{-1}L_N(\hat{x}_{N|N}^P-\hat{x}_{N|N}^W).\label{A.8}
\end{align}
By making use of (\ref{7}), (\ref{16}), (\ref{A.7}) and (\ref{A.8}), (\ref{13}) can be written as
\begin{align}
\nonumber\lambda_{N-1}&=A'P_{N+1}(A\hat{x}_{N|N}^P+Bu_N+B^P\tilde{u}_N^P)+Q\hat{x}_{N|N}^P\\
\nonumber             &=(A'P_{N+1}A-A'P_{N+1}B\Gamma_N^{-1}M_N+Q)\hat{x}_{N|N}^W+\hspace{-0.8mm}(A'\\
\nonumber             &\quad\times P_{N+1}A\hspace{-0.8mm}-\hspace{-0.8mm}A'P_{N+1}B^P\Omega_N^{-1}L_N\hspace{-0.8mm}+\hspace{-0.8mm}Q)(\hat{x}_{N|N}^P\hspace{-0.8mm}-\hspace{-0.8mm}\hat{x}_{N|N}^W)
\end{align}
Noting (\ref{19})-(\ref{25}), it can be known that (\ref{27}) holds for $k=N$. In order to accomplish the proof of the mathematical induction, let any $l$ with $0\leq l\leq N$. Assume that $\lambda_{k-1}$ are as (\ref{27}) for $k\geq l+1$. Now we shall prove that (\ref{27}) holds for $k=l$.

For $k=l+1$, (\ref{27}) is as
\begin{align}
\lambda_l=P_{l+1}^W\hat{x}_{l+1|l+1}^W+P_{l+1}^P(\hat{x}_{l+1|l+1}^P-\hat{x}_{l+1|l+1}^W).\label{A.9}
\end{align}
By making use of (\ref{6}) and (\ref{A.9}), (\ref{14}) becomes
\begin{align}
\nonumber0&=B'P_{l+1}^W(A\hat{x}_{l|l}^W+Bu_l)+Ru_l.
\end{align}
Using (\ref{19}) and (\ref{20}), we have the optimal $u_l$ as
\begin{align}
u_l=-\Gamma_l^{-1}M_l\hat{x}_{l|l}^W.\label{A.10}
\end{align}
By applying (\ref{A.2}), (\ref{A.4}) and (\ref{A.9}), (\ref{15}) can be calculated as
\begin{align}
\nonumber0&=B^{P'}P_{l+1}^W[(1-p)A(\hat{x}_{l|l}^P-\hspace{-0.8mm}\hat{x}_{l|l}^W)\hspace{-0.8mm}+(1-p)B^P\tilde{u}_l^P+\hspace{-0.8mm}A\hat{x}_{l|l}^W\\
\nonumber&\quad+Bu_l]+B^{P'}P_{l+1}^P[pA(\hat{x}_{l|l}^P-\hat{x}_{l|l}^W)+pB^P\tilde{u}_l^P]\\
\nonumber&\quad-B^{P'}P_{l+1}^W(A\hat{x}_{l|l}^W+Bu_l)+R^P\tilde{u}_l^P\\
\nonumber&=B^{P'}[(1-p)P_{l+1}^W+pP_{l+1}^P]A(\hat{x}_{l|l}^P-\hat{x}_{l|l}^W)\\
\nonumber&\quad+B^{P'}[(1-p)P_{l+1}^W+pP_{l+1}^P]B^P\tilde{u}_l^P+R^P\tilde{u}_l^P.
\end{align}
With (\ref{21})-(\ref{23}), the optimal $\tilde{u}_l^P$ is as
\begin{align}
\tilde{u}_l^P=-\Omega_l^{-1}L_l(\hat{x}_{l|l}^P-\hat{x}_{l|l}^W).\label{A.11}
\end{align}
Using (\ref{13}), (\ref{A.2}), (\ref{A.4}), (\ref{A.9}), (\ref{A.10}) and (\ref{A.11}), $\lambda_{l-1}$ can be calculated as
\begin{align}
\nonumber\lambda_{l-1}&=A'P_{l+1}^W[(1-p)A(\hat{x}_{l|l}^P\hspace{-0.8mm}-\hspace{-0.8mm}\hat{x}_{l|l}^W)\hspace{-0.8mm}+\hspace{-0.8mm}(1-p)B^P\tilde{u}_l^P+\hspace{-0.8mm}A\hat{x}_{l|l}^W\\
\nonumber&\quad+Bu_l]+A'P_{l+1}^P[pA(\hat{x}_{l|l}^P\hspace{-0.8mm}-\hspace{-0.8mm}\hat{x}_{l|l}^W)+pB^P\tilde{u}_l^P]+Q\hat{x}_{l|l}^P\\
\nonumber&=A'[(1-p)P_{l+1}^W+pP_{l+1}^P]A(\hat{x}_{l|l}^P\hspace{-0.8mm}-\hspace{-0.8mm}\hat{x}_{l|l}^W)+A'P_{l+1}^WA\hat{x}_{l|l}^W\\
\nonumber&\quad-A'[(1-p)P_{l+1}^W+pP_{l+1}^P]B^P\Omega_l^{-1}L_l(\hat{x}_{l|l}^P\hspace{-0.8mm}-\hspace{-0.8mm}\hat{x}_{l|l}^W)\\
\nonumber&\quad-A'P_{l+1}^WB\Gamma_l^{-1}M_l\hat{x}_{l|l}^W+Q\hat{x}_{l|l}^P\\
\nonumber&=\big\{A'[(1\hspace{-0.8mm}-p)P_{l+1}^W\hspace{-0.8mm}+pP_{l+1}^P]A\hspace{-0.8mm}-\hspace{-0.8mm}A'[(1\hspace{-0.8mm}-p)P_{l+1}^W\hspace{-0.8mm}+pP_{l+1}^P]\\
\nonumber&\quad\quad\times B^P\Omega_l^{-1}L_l+Q\big\}(\hat{x}_{l|l}^P\hspace{-0.8mm}-\hspace{-0.8mm}\hat{x}_{l|l}^W)\\
\nonumber&\quad+(A'P_{l+1}^WA-M_l'\Gamma_l^{-1}M_l+Q)\hat{x}_{l|l}^W.
\end{align}
By applying (\ref{19})-(\ref{25}), we have that (\ref{27}) is valid for $k=l$. Therefore, we have proven that (\ref{27}) holds for $k=N,\ldots,0.$
\end{proof}
\section{Proof of Theorem 2}
\begin{proof}
Under Assumptions 1 and 2, supposing that the system (\ref{28}) is stabilizable in the mean-square sense, we shall show that there exist the solutions $P^W$ and $P^P$ to the algebraic Riccati equations (\ref{30}) and (\ref{31}) such that $P^W>0$ and $\Delta>0$.

To make the time horizon $N$ explicit in the finite horizon case, we rewrite $P_k^W$, $P_k^P$, $M_k$, $\Gamma_k$, $\Delta_k$, $\Omega_k$ and $L_k$ in (\ref{19})-(\ref{25}) as $P_k^W(N)$, $P_k^P(N)$, $M_k(N)$, $\Gamma_k(N)$, $\Delta_k(N)$, $\Omega_k(N)$ and $L_k(N)$.

Combining the algebraic Riccati equations (\ref{30})-(\ref{31}) with the observation equation (\ref{2}), it can be known that they are uncorrelated with each other. Hence, we set $H=I$ and $v_k=0$. Then the observation equation (\ref{2}) becomes $y_k^P=x_k$. Accordingly, it is readily obtained that
\begin{align}
\hat{x}_{k|k}^P=x_k,\Sigma_{k|k}^P=0.\label{C.1}
\end{align}
Noting that the algebraic Riccati equations (\ref{30})-(\ref{31}) are uncorrelated with the initial value, we set $\mu=0$. Then the optimal performance (\ref{26}) (without the additive noise) becomes
\begin{align} J_N^*&=\mathbb{E}\big\{x_0'\big[P_0^W(N)\hat{x}_{0|0}^W+\hspace{-0.8mm}P_0^P(N)(\hat{x}_{0|0}^P-\hspace{-0.8mm}\hat{x}_{0|0}^W)\big]\big\}\label{C.4}\\
\nonumber      &=\mathbb{E}[\gamma_0x_0'P_0^W(N)x_0+x_0'P_0^P(N)(x_0-\gamma_0x_0)]\\
   &=\mathbb{E}[x_0'\Delta_0(N)x_0]\geq0.\label{C.2}
\end{align}
Then we have that
\begin{align}
J_N^*=\mathbb{E}[x_0'\Delta_0(N)x_0]\leq\mathbb{E}[x_0'\Delta_0(N+1)x_0]= J_{N+1}^*. \nonumber
\end{align}
Since the initial value $x_0$ is arbitrary, it yields that $\Delta_0(N)$ increases with respect to $N$. Next we shall show that $\Delta_0(N)$ is bounded. Noting that the system (\ref{28}) is stabilizable in the mean-square sense, then from the Definition 1, we have that
\begin{align}
\nonumber&\lim_{k\to\infty}\mathbb{E}(x_k'x_k)\\
\nonumber&=\lim_{k\to\infty}\mathbb{E}\{[(x_k-\hat{x}_{k|k}^P)+(\hat{x}_{k|k}^P-\hat{x}_{k|k}^W)+\hat{x}_{k|k}^W]'[(x_k-\hat{x}_{k|k}^P)\\
\nonumber&\qquad\qquad+(\hat{x}_{k|k}^P-\hat{x}_{k|k}^W)+\hat{x}_{k|k}^W]\}\\
\nonumber&=\lim_{k\to\infty}\big\{tr(\Sigma_{k|k}^P)+\mathbb{E}[(\hat{x}_{k|k}^P-\hat{x}_{k|k}^W)'(\hat{x}_{k|k}^P-\hat{x}_{k|k}^W)]\\
&\qquad\qquad+\mathbb{E}(\hat{x}_{k|k}^{W'}\hat{x}_{k|k}^W)\big\}=0.\label{C.3}
\end{align}
Combining (\ref{C.1}) with (\ref{C.3}), we have that
\begin{align}
\lim_{k\to\infty}\mathbb{E}(\hat{x}_{k|k}^{W'}\hat{x}_{k|k}^W)=0,\mathbb{E}[(\hat{x}_{k|k}^P-\hat{x}_{k|k}^W)'(\hat{x}_{k|k}^P-\hat{x}_{k|k}^W)]=0.\nonumber
\end{align}
From \cite{R30}, there exist constants $l_1>0$, $l_2>0$ and $l_3>0$ such that
\begin{align}
\nonumber&\sum_{k=0}^\infty\mathbb{E}(x_k'x_k)\leq l_1\mathbb{E}(x_0'x_0),\sum_{k=0}^\infty\mathbb{E}(\hat{x}_{k|k}^{W'}\hat{x}_{k|k}^W)\leq l_2\mathbb{E}(\hat{x}_{0|0}^{W'}\hat{x}_{0|0}^W)\\
\nonumber&\sum_{k=0}^\infty\mathbb{E}[(\hat{x}_{k|k}^P-\hat{x}_{k|k}^W)'(\hat{x}_{k|k}^P-\hat{x}_{k|k}^W)]\leq l_3\mathbb{E}[(\hat{x}_{0|0}^P-\hat{x}_{0|0}^W)'\\
\nonumber&\qquad\qquad\qquad\qquad\qquad\qquad\qquad\qquad\times(\hat{x}_{0|0}^P-\hat{x}_{0|0}^W)].
\end{align}
Noting Definition 2, let a constant $l_4$ such that $L^{W'}RL^W\leq l_4 I$, $L^{P'}RL^P\leq l_4 I$ and $Q\leq l_4 I$. Then the performance (\ref{29}) becomes
\begin{align}
\nonumber J&=\mathbb{E}\sum_{k=0}^\infty[{x_k}'Qx_k+u_{k}'Ru_k+\tilde{u}_{k}^{P'}R^P\tilde{u}_{k}^P]\\
\nonumber  &=\mathbb{E}\sum_{k=0}^\infty({x_k}'Qx_k)+\mathbb{E}\sum_{k=0}^\infty(\hat{x}_{k|k}^{W'}L^{W'}RL^W\hat{x}_{k|k}^W)\\
\nonumber  &\quad+\mathbb{E}\sum_{k=0}^\infty[(\hat{x}_{k|k}^P-\hat{x}_{k|k}^W)'L^{P'}R^PL^P(\hat{x}_{k|k}^P-\hat{x}_{k|k}^W)]\\
\nonumber  &\quad\leq l_4\{l_1\mathbb{E}(x_0'x_0)+l_2\mathbb{E}(\hat{x}_{0|0}^{W'}\hat{x}_{0|0}^W)\\
\nonumber&\quad\quad\quad+l_3\mathbb{E}[(\hat{x}_{0|0}^P-\hat{x}_{0|0}^W)'(\hat{x}_{0|0}^P-\hat{x}_{0|0}^W)]\}.
\end{align}
Thus, with (\ref{C.2}), we get
\begin{align}
\nonumber\mathbb{E}[x_0'\Delta_0(N)x_0]=J_N^*\leq J,
\end{align}
which means that $\Delta_0(N)$ is bounded. Hence, $\Delta_0(N)$ is convergent.

It is noted that the variables in (\ref{19})-(\ref{25}) are time invariant for $N$ due to the choice that $P_{N+1}=0$, i.e.,
\begin{align}
\nonumber P_k^W(N)&=P_{k-s}^W(N-s), P_k^P(N)=P_{k-s}^P(N-s),\\
\nonumber L_k(N)&=L_{k-s}(N-s), \Gamma_k=\Gamma_{k-s}(N-s),\\
\nonumber\Delta_k(N)&=\Delta_{k-s}(N-s), \Omega_k(N)=\Omega_{k-s}(N-s),\\
\nonumber M_k(N)&=M_{k-s}(N-s), s\leq k\leq N, 0\leq s\leq N.
\end{align}
Hence, it yields that
\begin{align}
\lim_{N\to\infty}\Delta_k(N)=\lim_{N\to\infty}\Delta_0(N-k)=\Delta.\nonumber
\end{align}
Thus, we have shown that $\Delta_k(N)$ is convergent. Now we shall prove that $P_k^W(N)$ and $P_k^P(N)$ are convergent respectively. Noting that $P_k^W(N)$ is uncorrelated with the packet dropout probability $p$, we set $p=0$. With (\ref{23}), it is readily obtained that $P_k^W(N)$ is convergent due to the convergence of $\Delta_k(N)$. Accordingly, from (\ref{23}), the convergence of $P_k^P(N)$ can been obtained for the convergence of $P_k^W(N)$ and $\Delta_k(N)$.

Finally, we shall show that there exists $l>0$ satisfying $\Delta_0(l)>0$. Assume this is not the case. Then there exists $x_0\neq0$ such that $\mathbb{E}(x_0'\Delta_0(N)x_0)=0$. The optimal performance (\ref{C.2}) becomes
\begin{align}
\nonumber J_N^*&=\sum_{k=0}^N\mathbb{E}[x_k^{*'}Qx_k^*+u_k^{*'}Ru_k+\tilde{u}_k^{L*}R^L\tilde{u}_k^L]\\
\nonumber      &=\mathbb{E}(x_0'\Delta_0(N)x_0)=0,
\end{align}
where $x_k^*$, $u_k^*$ and $\tilde{u}_k^*$ stand for the optimal state and optimal controllers respectively. From Assumption 1, i.e., $R>0$, $R^L>0$ and $Q=D'D\geq0$, we have that
\begin{align}
u_k^*=0,\tilde{u}_k^{L*}=0,Dx_k^*=0.\nonumber
\end{align}
Noting Assumption 2, i.e., ($A,Q^{\frac{1}{2}}$) is observable, it yields that $x_0=0$, which is a discrepancy of $x_0\neq 0$. Thus, there exists $l>0$ satisfying $\Delta_0(l)>0$. Hence, we have shown that $\Delta=\lim_{N\to\infty}\Delta_0(N)>0$. Similarly, we can obtain that $P^W>0$. Now the proof of Theorem 2 is finished.
\end{proof}
\section{Proof of Theorem 3}
\begin{proof}
``Sufficiency'': Under Assumptions 1 and 2, supposing that there exist solutions $P^W$ and $P^P$ to the algebraic Riccati equations (\ref{30}) and (\ref{31}) such that $P^W>0$ and $\Delta>0$, we shall show that the system (\ref{28}) is stabilizable in the mean-square sense.

Combining (\ref{27}) with the optimal performance (\ref{26}), we denote the Lyapunov function $\tilde{V}_k$ as
\begin{align}
\nonumber&\tilde{V}_k\\
\nonumber&=\mathbb{E}\bigg\{x_k'P^W\hat{x}_{k|k}^W\hspace{-0.8mm}+\hspace{-0.8mm}x_k'P^P(\hat{x}_{k|k}^P\hspace{-0.8mm}-\hspace{-0.8mm}\hat{x}_{k|k}^W)\hspace{-0.8mm}+\hspace{-0.8mm}\sum_{i=k}^\infty\{(x_i\hspace{-0.8mm}-\hspace{-0.8mm}\hat{x}_{i|i}^P)'\\
\nonumber&\quad\times[A'\Delta_{i+1}A\hspace{-0.8mm}+Q-(A-G_{i+1|i}^PHA)'P_{i+1}^P(A\\
&\quad-G_{i+1|i}^PHA)](x_i-\hspace{-0.8mm}\hat{x}_{i|i}^P)-pQ_vG_{i+1|i}^{P'}P_{i+1}^PG_{i+1|i}^P\}\bigg\}.\label{B.1}
\end{align}
Accordingly, we have
\begin{align}
\nonumber&\tilde{V}_k-\tilde{V}_{k+1}\\
\nonumber&=\mathbb{E}\bigg[x_k'P^W\hat{x}_{k|k}^W\hspace{-0.8mm}+\hspace{-0.8mm}x_k'P^P(\hat{x}_{k|k}^P\hspace{-0.8mm}-\hspace{-0.8mm}\hat{x}_{k|k}^W)\hspace{-0.8mm}+\hspace{-0.8mm}(x_k\hspace{-0.8mm}-\hspace{-0.8mm}\hat{x}_{k|k}^P)'\\
\nonumber&\quad\times[(1-\hspace{-0.8mm}p)A'P^WA\hspace{-0.8mm}+Q+\hspace{-0.8mm}pA'P^PG_{k+1|k}^PHA](x_k-\hspace{-0.8mm}\hat{x}_{k|k}^P)\\
\nonumber&\quad-x_{k+1}'P^W\hat{x}_{{k+1}|{k+1}}^W\hspace{-0.8mm}-\hspace{-0.8mm}x_{k+1}'P^P(\hat{x}_{{k+1}|{k+1}}^P\hspace{-0.8mm}-\hspace{-0.8mm}\hat{x}_{{k+1}|{k+1}}^W)\bigg]\\
\nonumber&=\mathbb{E}\big\{\hspace{-0.8mm}x_k'(P^W\hspace{-0.8mm}-\hspace{-0.8mm}A'P^WA\hspace{-0.8mm}+\hspace{-0.8mm}M'\Gamma^{-1}M)x_k\hspace{-0.8mm}-\hspace{-0.8mm}x_k'M'\Gamma^{-1}Mx_k\\
\nonumber&\quad+(\hat{x}_{k|k}^P\hspace{-0.8mm}-\hspace{-0.8mm}\hat{x}_{k|k}^W)'[P^P-P^W+pA'P^WA-pA'P^PA]\\
\nonumber&\quad\times(\hat{x}_{k|k}^P\hspace{-0.8mm}-\hspace{-0.8mm}\hat{x}_{k|k}^W)-2\tilde{u}_k^{P'}B^{P'}[pP^P+(1-p)P^W]A(\hat{x}_{k|k}^P\hspace{-0.8mm}\\
\nonumber&\quad-\hspace{-0.8mm}\hat{x}_{k|k}^W)-\tilde{u}_k^{P'}(\Omega-R^R)\tilde{u}_k^P-2u_k'B'P^WA\hat{x}_{k|k}^W\\
\nonumber&\quad-u_k'(\Gamma\hspace{-0.8mm}-\hspace{-0.8mm}R)u_k\big\}\\
\nonumber&=\mathbb{E}\big\{x_k'Qx_k+u_k'Ru_k+\tilde{u}_k^{P'}R^P\tilde{u}_k^P-(u_k+\Gamma^{-1}M\hat{x}_{k|k}^W)'\\
\nonumber&\quad\times\Gamma(u_k+\Gamma^{-1}M\hat{x}_{k|k}^W)-[\tilde{u}_k^P+\Omega^{-1}L(\hat{x}_{k|k}^P\hspace{-0.8mm}-\hspace{-0.8mm}\hat{x}_{k|k}^W)]'\\
&\quad\times\Omega[\tilde{u}_k^P+\Omega^{-1}L(\hat{x}_{k|k}^P\hspace{-0.8mm}-\hspace{-0.8mm}\hat{x}_{k|k}^W)]\big\}\label{B.8}\\
&=\mathbb{E}\big\{x_k'Qx_k+u_k'Ru_k+\tilde{u}_k^{P'}R^P\tilde{u}_k^P]\geq0,\label{B.2}
\end{align}
which implies that $\tilde{V}_k$ decreases with respect to $k$. Next we shall show that $\tilde{V}_k$ is bounded below.

Noting the optimal performance (\ref{26}) (without the additive noise) and the Lyapunov function $\tilde{V}_k$ (\ref{B.1}),
and letting the initial time $k\to\infty$, it can be readily obtained that $\lim_{k\to\infty}\tilde{V}_k\geq0$ which implies that $\tilde{V}_k$ is bounded below. Thus, $\tilde{V}_k$ is convergent.

Now select $m>0$. Taking summation for $k=m$ to $k=m+N$ on both sides of (\ref{B.2}) and letting $m\to\infty$, yielding
\begin{align}
\nonumber&\lim_{m\to\infty}\sum_{k=l}^{l+N}\mathbb{E}\big\{x_k'Qx_k+u_k'Ru_k+\tilde{u}_k^{P'}R^P\tilde{u}_k^P]\\
&=\lim_{m\to\infty}\tilde{V}_l-\tilde{V}_{l+N+1}=0,\label{B.3}
\end{align}
where (\ref{B.3}) holds for the convergence of $\tilde{V}_k$.

Noting the optimal performance (\ref{26}) (without the additive noise of the system), and choosing the initial value $x_0=0$, we have that $tr\hspace{-0.8mm}\sum_{k=0}^N\{\Sigma_{k|k}^P[(A'\Delta_{k+1} A\hspace{-0.8mm}+Q-p(A-G_{k+1|k}^PHA)'P_{k+1}^P(A-G_{k+1|k}^PHA)]-pQ_vG_{k+1|k}^{P'}P_{k+1}^PG_{k+1|k}^P\}\geq0$. Thus, we have that
\begin{align}
\nonumber&\sum_{k=0}^{N}\mathbb{E}\big\{x_k'Qx_k+u_k'Ru_k+\tilde{u}_k^{P'}R^P\tilde{u}_k^P]\geq\mathbb{E}\big\{x_0'\big[P_0^W\hat{x}_{0|0}^W\\
\nonumber&+\hspace{-0.8mm}P_0^P(\hat{x}_{0|0}^P-\hspace{-0.8mm}\hat{x}_{0|0}^W)\big]\big\}\hspace{-0.8mm}+\hspace{-0.8mm}\sum_{k=0}^Ntr\bigg\{\Sigma_{k|k}^P[A'\Delta_{k+1}A\hspace{-0.8mm}+Q-p(A\\
\nonumber&-G_{k+1|k}^PHA)'P_{k+1}^P(A-G_{k+1|k}^PHA)]\hspace{-0.8mm}-pQ_vG_{k+1|k}^{P'}P_{k+1}^P\\
\nonumber&\times G_{k+1|k}^P\geq\mathbb{E}\big\{x_0'\big[P_0^W\hat{x}_{0|0}^W+\hspace{-0.8mm}P_0^P(\hat{x}_{0|0}^P-\hspace{-0.8mm}\hat{x}_{0|0}^W)\big]
\end{align}

Setting $\sigma=0$, with Lemma 1, the above equation becomes
\begin{align}
\nonumber&\sum_{k=0}^{N}\mathbb{E}\big\{x_k'Qx_k+u_k'Ru_k+\tilde{u}_k^{P'}R^P\tilde{u}_k^P]\\
\nonumber&\geq\hspace{-0.8mm}\mathbb{E}\{x_0'[(1-p)P_0^W+p\mu)]x_0\}.
\end{align}
Through a time-shift of length of $m$, letting $m\to\infty$ and noting (\ref{B.3}), it yields
\begin{align}
\nonumber&\lim_{m\to\infty}\sum_{k=m}^{N+m}\mathbb{E}\big\{x_k'Qx_k+u_k'Ru_k+\tilde{u}_k^{P'}R^P\tilde{u}_k^P]\\
\nonumber&\geq\hspace{-0.8mm}\lim_{m\to\infty}\mathbb{E}\{x_m'[(1-p)P_m^W+p\mu)]x_m\}=0.
\end{align}
Noting $P_k^W>0$, we have that $\lim_{k\to\infty}\mathbb{E}\{x_k'x_k\}=0$.
Thus, the system (\ref{28}) can be stabilized in the mean-square sense by the controllers (\ref{37}) and (\ref{38}). Now we shall show that the controllers (\ref{37}) and (\ref{38}) can minimize the infinite-horizon performance (\ref{29}).

Taking summation on both sides of (\ref{B.8}) from $k=0$ to $k=\infty$ and noting the convergence of $\tilde{V}_k$, the infinite-horizon performance (\ref{29}) can be written as
\begin{align}
\nonumber J&=\tilde{V}_0+\mathbb{E}\sum_{k=0}^\infty\{(u_k+\Gamma^{-1}M\hat{x}_{k|k}^W)'\Gamma(u_k+\Gamma^{-1}M\hat{x}_{k|k}^W)\\
\nonumber&\quad-[\tilde{u}_k^P\hspace{-0.8mm}+\hspace{-0.8mm}\Omega^{-1}L(\hat{x}_{k|k}^P\hspace{-0.8mm}-\hspace{-0.8mm}\hat{x}_{k|k}^W)]'\Omega[\tilde{u}_k^P\hspace{-0.8mm}+\Omega^{-1}L(\hat{x}_{k|k}^P\hspace{-0.8mm}-\hspace{-0.8mm}\hat{x}_{k|k}^W)]\}
\end{align}
Since $\Gamma>0$ and $\Omega>0$, the stabilizing controllers (\ref{37}) and (\ref{38}) can also minimize (\ref{29}), and the optimal performance is as (\ref{39}). This completes the proof of the sufficiency. The proof of the necessity has been given in Appendix B.
\end{proof}
\section{Proof of Theorem 4}
\begin{proof}
Under assumption 1 and 2, if $\sqrt{p}|\lambda_{max}(A\hspace{-0.8mm}-\hspace{-0.8mm}B^P\Omega^{-1}L)|<1$, assuming that there exist solutions $P^W$ and $P^P$ to the algebraic Riccati equations (\ref{30}) and (\ref{31}) such that $P^W>0$ and $\Delta>0$, we shall show that the system (\ref{7}) is bounded in the mean-square sense.

To begin with, we shall give some preliminary work as follows:
\begin{align}
\nonumber&\mathbb{E}[(\hat{x}_{k|k}^P-\hat{x}_{k|k}^W)'(\hat{x}_{k|k}^P-\hat{x}_{k|k}^W)]\\
\nonumber&=\mathbb{E}\{[(x_k\hspace{-0.8mm}-\hspace{-0.8mm}\hat{x}_{k|k}^W)\hspace{-0.8mm}-\hspace{-0.8mm}(x_k\hspace{-0.8mm}-\hspace{-0.8mm}\hat{x}_{k|k}^P)]'[(x_k\hspace{-0.8mm}-\hspace{-0.8mm}\hat{x}_{k|k}^W)\hspace{-0.8mm}-\hspace{-0.8mm}(x_k\hspace{-0.8mm}-\hspace{-0.8mm}\hat{x}_{k|k}^P)]\}\\
\nonumber&=tr(\Sigma_{k|k}^W)\hspace{-0.8mm}-\hspace{-0.8mm}\mathbb{E}[x_k'(x_k-\hat{x}_{k|k}^P)]\hspace{-0.8mm}-\hspace{-0.8mm}\mathbb{E}[(x_k\hspace{-0.8mm}-\hspace{-0.8mm}\hat{x}_{k|k}^P)'x_k]\hspace{-0.8mm}+\hspace{-0.8mm}tr(\Sigma_{k|k}^P)\\
\nonumber&=tr(\Sigma_{k|k}^W-\Sigma_{k|k}^P-\Sigma_{k|k}^P+\Sigma_{k|k}^P)\\
         &=tr(\Sigma_{k|k}^W-\Sigma_{k|k}^P)\label{D.1}
\end{align}
\begin{align}
\nonumber&\mathbb{E}[(x_k-\hat{x}_{k|k}^W)'(\hat{x}_{k|k}^P-\hat{x}_{k|k}^W)]\\
\nonumber&=\mathbb{E}[x_k'(\hat{x}_{k|k}^P-\hat{x}_{k|k}^W)]=\mathbb{E}\{x_k'[(x_k-\hat{x}_{k|k}^W)-(x_k-\hat{x}_{k|k}^P)]\}\\
         &\qquad\qquad\qquad\qquad\quad=tr(\Sigma_{k|k}^W-\Sigma_{k|k}^P).\label{D.2}
\end{align}
\begin{align}
\nonumber&\mathbb{E}[x_k'(\hat{x}_{k|k}^P-\hat{x}_{k|k}^W)]\\
\nonumber&=\mathbb{E}\{[(x_k-\hat{x}_{k|k}^P)+(\hat{x}_{k|k}^P-\hat{x}_{k|k}^W)+\hat{x}_{k|k}^W]'(\hat{x}_{k|k}^P-\hat{x}_{k|k}^W)\}\\
\nonumber&=tr(\Sigma_{k|k}^P+\Sigma_{k|k}^W-\Sigma_{k|k}^P)\\
         &=tr(\Sigma_{k|k}^W).\label{D.6}
\end{align}
In virtue of (\ref{42}) and (\ref{43}), we have
\begin{align}
\nonumber x_{k+1}&=Ax_k-B\Gamma^{-1}M\hat{x}_{k|k}^W-B^P\Omega^{-1}L(\hat{x}_{k|k}^P-\hat{x}_{k|k}^W)+\omega_k\\
\nonumber        &=Ax_k+B\Gamma^{-1}M[(x_k-\hat{x}_{k|k}^W)-x_k]\\
\nonumber        &\quad-B^P\Omega^{-1}L(\hat{x}_{k|k}^P-\hat{x}_{k|k}^W)+\omega_k\\
\nonumber        &=(A-B\Gamma^{-1}M)x_k+B\Gamma^{-1}M(x_k-\hat{x}_{k|k}^W)\\
                 &\quad-B^P\Omega^{-1}L(\hat{x}_{k|k}^P-\hat{x}_{k|k}^W)+\omega_k.\label{D.7}
\end{align}
Using (\ref{D.1}), (\ref{D.2}), (\ref{D.6}) and (\ref{D.7}), it yields that
\begin{align}
\nonumber&\mathbb{E}[x_{k+1}'x_{k+1}]\\
\nonumber&=\mathbb{E}[x_k'(A\hspace{-0.8mm}-\hspace{-0.8mm}B\Gamma^{-1}M)'(A\hspace{-0.8mm}-\hspace{-0.8mm}B\Gamma^{-1}M)x_k]+\hspace{-0.8mm}2tr[\Sigma_{k|k}^W(A-B\\
\nonumber&\quad\times\Gamma^{-1}M)B\Gamma^{-1}M-2\Sigma_{k|k}^W(A-B\Gamma^{-1}M)'B^P\Omega^{-1}L\\
\nonumber&\quad-2(\Sigma_{k|k}^W-\Sigma_{k|k}^P)M'\Gamma^{-1}B'B^P\Omega^{-1}L+\Sigma_{k|k}^WM'\Gamma^{-1}B'\\
&\quad\times B\Gamma^{-1}M+(\Sigma_{k|k}^W-\Sigma_{k|k}^P)L'\Omega^{-1}B^{P'}B^P\Omega^{-1}L+Q_\omega].\label{D.10}
\end{align}
Noting Lemma 5, i.e., $\Sigma_{k|k}^W$ and $\Sigma_{k|k}^P$ are convergent. Thus,  the second term of equation (\ref{D.10}) is convergent obviously.
Hence, it can be known that $\lim_{k\to}\mathbb{E}(x_k'x_k)$ is bounded in the mean-square sense if and only if the following linear system
\begin{align}
\beta_{k+1}=(A-B\Gamma^{-1}M)\beta_k,\label{D.3}
\end{align}
with the initial value $\beta_0=x_0$, is stable in the mean-square sense.

Noting (\ref{30}), (\ref{32}) and (\ref{33}), (\ref{30}) can be written as
\begin{align}
\nonumber P^W&=M'\Gamma^{-1}R\Gamma^{-1}M+Q\\
             &\quad+(A-B\Gamma^{-1}M)'P^W(A-B\Gamma^{-1}M).\label{D.4}
\end{align}
Now we shall show that (\ref{D.3}) is stable in the mean-square sense. Denote the Lyapunov function $W_k$ as
\begin{align}
\nonumber W_k=\mathbb{E}(\beta_k'P^W\beta_k).
\end{align}
By making use of (\ref{D.4}), we get
\begin{align}
\nonumber &W_{k+1}-W_k\\
\nonumber &=\mathbb{E}\{\beta_k'[(A-B\Gamma^{-1}M)'P^W(A-B\Gamma^{-1}M)-P^W]\beta_k\}\\
\nonumber &=-\mathbb{E}[\beta_k'(M'\Gamma^{-1}R\Gamma^{-1}M+Q)\beta_k],
\end{align}
which implies that $W_k$ decreases with respect to $k$ and bounded below, i.e, $W_k$ is convergent. Adding from $k=0$ to $k=m$ on both sides of the above equation, we have
\begin{align}
\nonumber &W_{m+1}-W_0=-\sum_{k=0}^m\mathbb{E}[\beta_k'(M'\Gamma^{-1}R\Gamma^{-1}M+Q)\beta_k].
\end{align}
Letting $m\to\infty$ on both sides of the above equation, we get
\begin{align}
\nonumber&\lim_{m\to\infty}\mathbb{E}(\beta_{m+1}'P^W\beta_{m+1})\\
\nonumber&=\mathbb{E}(\beta_0'P^W\beta_0)-\lim_{m\to\infty}\sum_{k=0}^m\mathbb{E}[\beta_k'(M'\Gamma^{-1}R\Gamma^{-1}M+Q)\beta_k].
\end{align}
Due to the convergence of $W_k$, it can be obtained that $\lim_{m\to\infty}\mathbb{E}[\beta_m'(M'\Gamma^{-1}R\Gamma^{-1}M+Q)\beta_m]=0$. Thus, $\lim_{k\to\infty}\mathbb{E}[\beta_k'\beta_k]=0$, i.e., the system (\ref{D.3}) is stable in the mean-square sense. Hence, the system (\ref{7}) is bounded in the mean-square sense.

Now we shall show that (\ref{42}) and (\ref{43}) minimize the performance (\ref{41}). Denote
\begin{align}
\nonumber\tilde{W}_k&=\mathbb{E}[x_k'P^W\hat{x}_{k|k}^W\hspace{-0.8mm}+\hspace{-0.8mm}x_k'P^P(\hat{x}_{k|k}^P\hspace{-0.8mm}-\hspace{-0.8mm}\hat{x}_{k|k}^W)]\hspace{-0.8mm}+\hspace{-0.8mm}\mathbb{E}\sum_{i=k}^\infty\{(x_i\hspace{-0.8mm}-\hspace{-0.8mm}\hat{x}_{i|i}^P)'\\
\nonumber&\quad\times[(A'\Delta A\hspace{-0.8mm}+Q-(A-G^PHA)'P^P(A-G^PHA)](x_i\\
\nonumber&\quad-\hspace{-0.8mm}\hat{x}_{i|i}^P)\}+\sum_{i=k}^\infty tr\{Q_\omega[(\Delta_{k+1}-p(I\hspace{-0.8mm}-G_{k+1|k}^PH)'P_{k+1}^P\\
\nonumber&\quad\times(I-G_{k+1|k}^PH)]-pQ_vG_{k+1|k}^{P'}P_{k+1}^PG_{k+1|k}^P\}.
\end{align}
Similar to (\ref{B.8}), it yields that
\begin{align}
\nonumber&\tilde{W}_k-\tilde{W}_{k+1}\\
\nonumber&=\mathbb{E}\big\{x_k'Qx_k+u_k'Ru_k+\tilde{u}_k^{P'}R^P\tilde{u}_k^P-(u_k+\Gamma^{-1}M\hat{x}_{k|k}^W)'\\
\nonumber&\quad\times\Gamma(u_k+\Gamma^{-1}M\hat{x}_{k|k}^W)-[\tilde{u}_k^P+\Omega^{-1}L(\hat{x}_{k|k}^P\hspace{-0.8mm}-\hspace{-0.8mm}\hat{x}_{k|k}^W)]'\\
\nonumber&\quad\times\Omega[\tilde{u}_k^P+\Omega^{-1}L(\hat{x}_{k|k}^P\hspace{-0.8mm}-\hspace{-0.8mm}\hat{x}_{k|k}^W)]\big\}.
\end{align}
Noting Lemma 5 and the bounedness in the mean-square sense of the system (\ref{7}), it can be obtained that $\lim_{k\to\infty}\tilde{W}_k$ is bounded. Adding from $k=0$ to $k=N$ on both sides of the above equation, the performance (\ref{41}) becomes
\begin{align}
\nonumber\tilde{J}&=\lim_{N\to\infty}\frac{1}{N}\bigg\{\tilde{W}_0-\tilde{W}_{N+1}+\sum_{k=0}^N\{(u_k+\Gamma^{-1}M\hat{x}_{k|k}^W)'\\
\nonumber&\quad\times\Gamma(u_k+\Gamma^{-1}M\hat{x}_{k|k}^W)+[\tilde{u}_k^P+\Omega^{-1}L(\hat{x}_{k|k}^P\hspace{-0.8mm}-\hspace{-0.8mm}\hat{x}_{k|k}^W)]'\\
\nonumber&\quad\times\Omega[\tilde{u}_k^P+\Omega^{-1}L(\hat{x}_{k|k}^P\hspace{-0.8mm}-\hspace{-0.8mm}\hat{x}_{k|k}^W)]\}\bigg\}\\
\nonumber&=(u_k+\Gamma^{-1}M\hat{x}_{k|k}^W)'\Gamma(u_k+\Gamma^{-1}M\hat{x}_{k|k}^W)+[\tilde{u}_k^P\\
\nonumber&\quad+\Omega^{-1}L(\hat{x}_{k|k}^P\hspace{-0.8mm}-\hspace{-0.8mm}\hat{x}_{k|k}^W)]'\Omega[\tilde{u}_k^P+\Omega^{-1}L(\hat{x}_{k|k}^P\hspace{-0.8mm}-\hspace{-0.8mm}\hat{x}_{k|k}^W)]\\
\nonumber&\quad+tr\{\Sigma^P[(A'\Delta A\hspace{-0.8mm}+Q-(A-G^PHA)'P^P(A\\
\nonumber&\quad-G^PHA)]+Q_\omega[(\Delta_{k+1}-p(I\hspace{-0.8mm}-G_{k+1|k}^PH)'P_{k+1}^P\\
\nonumber&\quad\times(I-G_{k+1|k}^PH)]-pQ_vG_{k+1|k}^{P'}P_{k+1}^PG_{k+1|k}^P\}.
\end{align}
Noting that $\Gamma>0$ and $\Omega>0$, it can be readily obtained that the optimal controllers are as (\ref{42}) and (\ref{43}). Accordingly, the optimal performance is as (\ref{44}). The sufficiency of Theorem 4 is completed.

``Necessity'': Suppose that the system (\ref{7}) is bounded in the mean-square sense. we shall prove that there exist solutions $P^W$ and $P^P$ to the algebraic Riccati equations (\ref{30}) and (\ref{31}) such that $P^W>0$ and $\Delta>0$.

Substituting (\ref{37}) and (\ref{38}) into the system (\ref{28}) and  replacing $x_k$ with $z_k$, it yields
\begin{align}
\nonumber z_{k+1}&=Az_k-B\Gamma^{-1}M\hat{z}_{k|k}^W-B^P\Omega^{-1}L(\hat{z}_{k|k}^P-\hat{z}_{k|k}^W)\\
\nonumber        &=Az_k+B\Gamma^{-1}M[(z_k-\hat{z}_{k|k}^W)-z_k]\\
\nonumber        &\quad-B^P\Omega^{-1}L(\hat{z}_{k|k}^P-\hat{z}_{k|k}^W)\\
\nonumber        &=(A-B\Gamma^{-1}M)z_k+B\Gamma^{-1}M(z_k-\hat{z}_{k|k}^W)\\
        &\quad-B^P\Omega^{-1}L(\hat{z}_{k|k}^P-\hat{z}_{k|k}^W),\label{D.8}
\end{align}
with initial value $z_0=x_0$. Define $s_k$ as the following equation
\begin{align}
\nonumber s_{k+1}&=(A-B\Gamma^{-1}M)s_k+B\Gamma^{-1}M(s_k-\hat{s}_{k|k}^W)\\
                 &\quad-B^P\Omega^{-1}L(\hat{s}_{k|k}^P-\hat{s}_{k|k}^W)+\omega_k.\label{D.9}
\end{align}
with known initial value $s_k=0$.

Noting (\ref{7}), (\ref{8}) and (\ref{9}), it can be obtained that $x_k=z_k+s_k$. Through simple calculation, it can be known that $s_k$ is orthogonal to $z_k$. Thus, it can be readily obtained that ``the system (\ref{D.7}) is bounded in the mean-square sense'' is equivalent to ``the system (\ref{8}) is stabilizable in the mean-square sense''. From Theorem 3, if the system (\ref{28}) is stable in the mean-square sense, then there exist solutions $P^W$ and $P^P$ to the algebraic Riccati equations (\ref{30}) and (\ref{31}) such that $P^W>0$ and $\Delta>0$. Thus, the same conclusion can be obtained if the system (\ref{7}) is bounded in the mean-square sense. This completes the proof of the necessity.
\end{proof}


\end{document}